\documentclass{article}

\usepackage{color}
\usepackage[margin=1cm]{caption}

\usepackage{amssymb,amsmath,amsfonts}
\usepackage{amsthm}
\usepackage{cite} 
\usepackage{xspace}
\usepackage{hyperref}
\usepackage{graphicx}
\usepackage{subfig}

\newcommand{\figref}[1]{Fig.~\ref{fig:#1}}
\newcommand{\lemref}[1]{Lemma~\ref{lem:#1}}
\newcommand{\exref}[1]{Example~\ref{ex:#1}}
\newcommand{\eqnref}[1]{Eq.~\eqref{eqn:#1}}
\newcommand{\secref}[1]{Section~\ref{sec:#1}}
\newcommand{\ssecref}[1]{Section~\ref{ssec:#1}}
\newcommand{\subfigref}[2]{Fig.~\ref{fig:#1}\subref{subfig:#1-#2}}
\newcommand{\thmref}[1]{Theorem~\ref{thm:#1}}
\newcommand{\defref}[1]{Definition~\ref{def:#1}}

\newcommand{\corref}[1]{Corollary~\ref{cor:#1}}
\newcommand{\stmtref}[1]{Statement~\emph{\ref{stmt:#1}}} 
\newcommand{\defstmtref}[2]{\stmtref{#1-#2} of \defref{#1}}
\newcommand{\lemstmtref}[2]{\stmtref{#1-#2} of \lemref{#1}}

\DeclareMathOperator{\lkmap}{lk}

\DeclareMathOperator{\reach}{reach}
\DeclareMathOperator{\vertmap}{verts}
\newcommand{\cubes}[1]{\mathcal{#1}}
\newcommand{\clcubes}[1]{\overline{\cubes{#1}}}
\newcommand{\di}[1]{\dipaths{#1}{}{}}
\newcommand{\dipaths}[3]{\overrightarrow{P}_{#2}^{#3}(#1)}
\newcommand{\downset}[2]{#1_{\preceq#2}}
\renewcommand{\int}{\textsc{int}}
\newcommand{\Khat}{{\widehat{K}}}
\newcommand{\pastlk}[2]{\lkmap^{-}_{#1}(#2)}
\newcommand{\reachcplx}[2]{\reach(#1,#2)}

\newcommand{\upset}[2]{#1_{#2\preceq}}
\newcommand{\verts}[1]{{\vertmap({#1})}}

\newcommand{\vecj}{\mathbf{j}}
\newcommand{\veck}{\mathbf{k}}
\newcommand{\vecp}{\mathbf{p}}
\newcommand{\vecq}{\mathbf{q}}

\newcommand{\vecv}{\mathbf{v}}
\newcommand{\vecw}{\mathbf{w}}
\newcommand{\vecx}{\mathbf{x}}
\newcommand{\vecy}{\mathbf{y}}
\newcommand{\zero}{{\mathbf{0}}} 
\newcommand{\setzero}{\{ \zero \}} %
\newcommand{\one}{\mathbf{1}} 
\newcommand{\binone}{\mathbf{1}} 
\newcommand{\binzero}{\mathbf{0}} 
\newcommand{\B}{\mathbb{B}}
\newcommand{\R}{\mathbb{R}}
\renewcommand{\S}{\mathbb{S}}
\newcommand{\Z}{\mathbb{Z}}

\theoremstyle{definition}
\newtheorem{theorem}{Theorem}
\newtheorem{definition}{Definition}
\newtheorem{remark}{Remark}
\newtheorem{lemma}{Lemma}
\newtheorem{corollary}{Corollary}
\newtheorem{example}{Example}
\newtheorem{acknowledgement}{Acknowledgement}

\begin{document}

\title{Combinatorial Conditions for Directed Collapsing}
\author{Robin Belton\footnote{ Robin Belton, Montana State University, Bozeman, MT, USA,
        robin.belton@montana.edu},
       Robyn Brooks\footnote{ Robyn Brooks, Boston College, Chestnut Hill, MA, USA,
        robyn.brooks@bc.edu},
       Stefania Ebli\footnote{ Stefania Ebli, \'{E}cole Polytechnique F\'{e}d\'{e}ral de Laussane, Lausanne, Switzerland,
        stefania.ebli@epfl.ch},
       Lisbeth Fajstrup\footnote{ Lisbeth Fajstrup, Aalborg University, Aalborg, Denmark,
        fajstrup@math.aau.dk},\\
       Brittany Terese Fasy\footnote{ Brittany Terese Fasy, Montana State University Bozeman, MT, USA,
        brittany.fasy@montana.edu},
       Nicole Sanderson\footnote{ Nicole Sanderson, Lawrence Berkeley National Laboratory, Berkeley, CA, USA, nzs5677@psu.edu},
       and Elizabeth Vidaurre\footnote{ Elizabeth Vidaurre, Molloy College, Rockville Centre, NY, USA,
        evidaurre@molloy.edu}}

%

\maketitle


\abstract{
The purpose of this article is to study directed collapsibility of directed Euclidean cubical complexes. One application of this is in the nontrivial task of verifying the execution of concurrent programs.
The classical definition of collapsibility involves certain conditions on a pair of cubes of the complex. 
The direction of the space can be taken into account by requiring that the past  links of vertices remain homotopy equivalent after collapsing. We call this type of collapse a \emph{link-preserving directed collapse}.
In this paper, we give combinatorially equivalent conditions for preserving the topology of the links, allowing for the implementation of an algorithm for collapsing a directed Euclidean cubical complex.
Furthermore, we give conditions for when link-preserving directed collapses preserve the contractability 
and connectedness of directed path spaces, as well as examples when link-preserving directed 
collapses do not preserve the number of connected components of the path space between the minimum
and a given vertex.
}

\section{Introduction}

A directed Euclidean cubical complex is a subset of $\R^n$ comprising a finite
union of directed unit cubes.  Directed paths (i.e., paths that are nondecreasing in
all coordinates) and spaces of directed paths are the objects of study in this
paper.
In particular, we address the question of how to simplify
directed Euclidean complexes without significantly changing the spaces of
directed~paths.

This model is motivated by several applications, where each axis of the model
corresponds to a parameter of the application (e.g., time).
In particular, Euclidean cubical complexes are used to model
concurreny in computer programming~\cite{dijkstra1977two,lisbeth,FGR,ziemianski2016execution}, hybrid
dynamical systems~\cite{wisniewski2006towards},
and motion planning~\cite{ghrist2010configuration}.
Consider the application to concurrency.  In this example, each
axis represents a sequence of actions a process completes in the program
execution. The complex itself corresponds to ``compatible'' parameters (i.e.,
when the processes can execute simultaneously).  Cubes missing from the complex
correspond to parameters for which the processes
cannot execute simultaneously for some reason, such as when they require the same
resources with limited capacity; see \figref{annulus-cube}.
A directed path (dipath) in the complex represents a, possibly
partial, program execution. Such executions are equivalent if the corresponding
dipaths are \emph{directed homotopic}.  Simplifying the complexes allows for a more compact
representation of the execution space, which, in turn, reduces the complexity
of validating correctness of concurrent programs.
\begin{figure}[th]
    \centering
    {\includegraphics[height=1.5in]{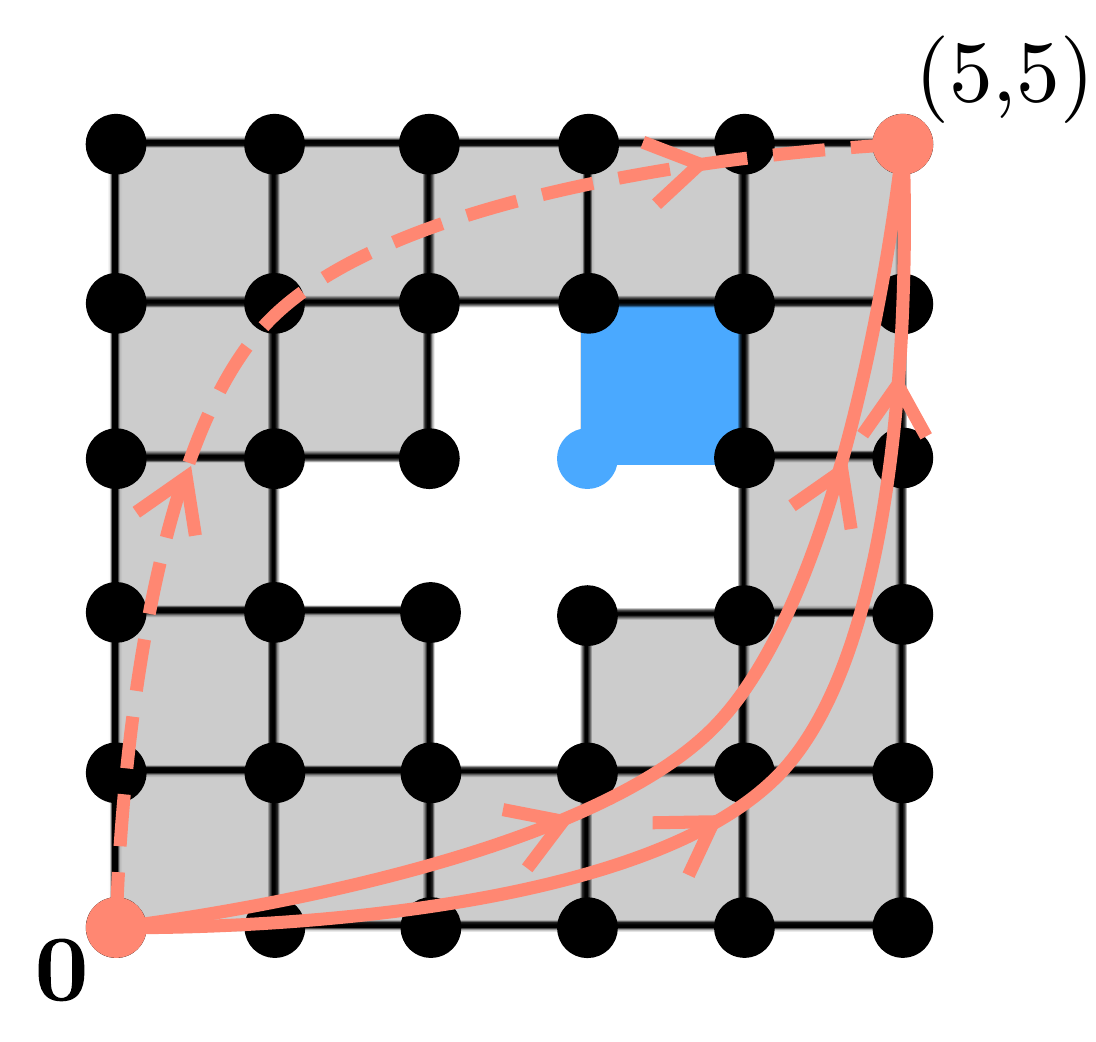}}
    \caption{The Swiss Flag and Three Directed Paths. The gray and blue
    squares are the two-cubes of a Euclidean cubical complex.
    The bi-monotone increasing
    paths are directed paths
    starting at $(0,0)$ and ending at~$(5,5)$.
    This complex has a cross-shaped hole in the middle.
    As a consequence, the solid directed paths are directed homotopic
    while the dashed directed path is not directed homotopic to either of the
    other directed paths. Each point highlighted in blue
    is \emph{unreachable}, meaning that we cannot reach any point highlighted in blue
    without breaking bi-monotonicity in a path starting at $(0,0)$.
    This complex models the dining philsophers problem, a well-known example
    in concurrency, where two processes require two shared
    resources with limited capacity~\cite{dijkstra1977two,hoare}.
    The two distinct
    paths (solid and dashed) represent which process uses both shared
    resources~first.
    }\label{fig:annulus-cube}
\end{figure}

A non-trivial Euclidean cubical complex
contains uncountably many dipaths and more information
than we need for understanding the topology of the spaces of dipaths. The main question
we ask is, \emph{How can we simplify a directed Euclidean cubical complex while still preserving spaces
of dipaths}?

\emph{Past links} are local representations of a Euclidean cubical complex at vertices. They were
introduced in~\cite{ziemianski2016execution} as a means to show that any finite homotopy
type can be realized as a connected component of the space of execution paths for some $PV$-model.
In~\cite{belton2020towards}, we found conditions for when the
local information of past links preserve the global information on the homotopy type of spaces
of dipaths. Because of these relationships between past links and dipath spaces,
we define collapsing in terms of past links.  We call this type of collapsing
\emph{link-preserving directed collapse} (LPDC). We aim to compress a Euclidean cubical
complex by LPDCs before attempting to answer questions about dipath spaces.

The main result of this paper is~\thmref{collapsingpairs},
which provides a simple criterion for such a
collapsing to be allowed:
\emph{A pair of cubes $(\tau,\sigma)$ is an LPDC pair  if and only if it is a
collapsing pair in the non-directed sense and~$\tau$ does not contain the
minimum vertex of $\sigma$.}
This condition greatly simplifies the definition of LPDC and  is easy to add to a collapsing
algorithm for Euclidean cubical complexes in the undirected setting. Algorithms and
implementations in this setting already exist such as in \cite{lachaud}. Furthermore, we provide
conditions for when LPDCs preserve the contractability and connectedness of dipath
spaces (\secref{pathspaces}) along with a discussion of some of the limitations (\secref{discussion}).
This work provides a start at the mathematical foundations for developing polynomial time
algorithms that collapse Euclidean cubical complexes and preserve dipath~spaces.

\section{Background}\label{sec:background}

This paper builds on our prior work~\cite{belton2020towards}, as well as work by
others~\cite{ziemianski2016execution,RZ,FGR,grandis2003directed,grandis2009diretced}.
In this section, we recall the definitions of directed Euclidean cubical
complexes, which are the objects that we study in this paper.
Then, we discuss the relationship between
spaces of directed paths and past links in directed Euclidean cubical complexes.
For additional background on directed topology (including generalizations of the
definitions below), we refer the reader to~\cite{lisbeth}.
We also assume the reader is familiar with the notion of homotopy equivalence of
topological spaces (denoted using $\simeq$ in this paper)
and homotopy between paths as presented in~\cite{hatcher}.

\subsection{Directed Spaces and Euclidean Cubical Complexes}

Let $n$ be a positive integer.
A \emph{(closed) elementary cube} in $\R^n$ is a product of closed intervals of the
following form:
\begin{equation}\label{eqn:elem-cube}
    [v_1-j_1,v_1] \times [v_2-j_2, v_2] \times
    \ldots \times [v_n-j_n, v_n]
    \subsetneq \R^n,
\end{equation}
where $\vecv=(v_1,v_2, \ldots, v_n) \in \Z^n$ and $\vecj=(j_1,j_2, \ldots, j_n) \in
\{0,1\}^n$.
We often refer to elementary cubes simply as \emph{cubes}.
The dimension of the cube is the number of unit entries in the vector $\vecj$;
specifically, the dimension of the cube in \eqnref{elem-cube}
is the sum:~$\sum_{i=1}^n j_i$. 

 In particular, when $\vecj=\zero := (0, 0,\ldots,0)$, the
elementary cube is a single point and often denoted using just~$\vecv$.
If~$\tau$ and~$\sigma$ are elementary cubes such that $\tau \subseteq \sigma$, we say
that $\tau$ is a face of~$\sigma$ and that~$\sigma$ is a coface of $\tau$.
Cubical sets were first introduced in the 1950s by Serre~\cite{serre1951homologie} in a more general
setting; see also~\cite{BrownOn81,jardine2002cubical,goubault1995geometry}.
\begin{figure}[th]
    \centering
    {\includegraphics[height=1.5in]{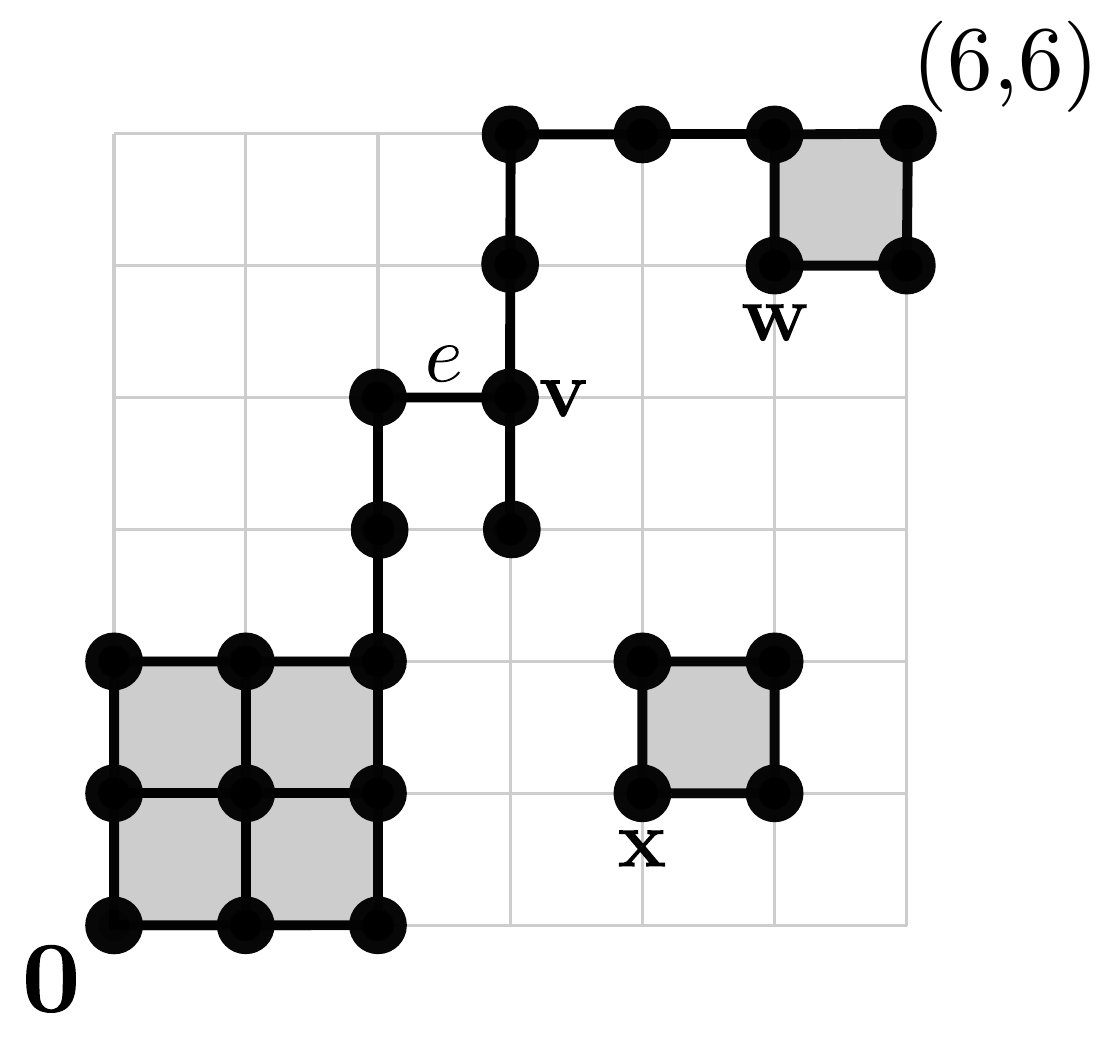}}
    \caption{Euclidean cubical complex in $\R^2$ with $24$
    zero-cubes (vertices),~$28$ one-cubes (edges),
    and six two-cubes (squares). By construction, all elementary cubes in a
    directed Euclidean cubical complex
    are axis aligned.  Consider the vertex $\vecv=(3,4)$.  The edge
    $e=[(2,4),(3,4)]$ (written $e=[2,3] \times [4,4]$ in the notation of
    \eqnref{elem-cube}) is one of the two lower cofaces of $\vecv$.  Since $e$
    is not a face of any two-cube, $e$ is a maximal cube (since it is not a face
    of a higher-dimensional cube).
    }\label{fig:ecc}
\end{figure}%

Elementary cubes stratify~$\R^n$, where two points~$x,y\in\R^n$ are in the same
stratum if and only if they are members of the same set of elementary cubes; we
call this the \emph{cubical stratification of~$\R^n$}. Each stratum in the
stratification is either an open cube or a single point.
A \emph{Euclidean cubical
complex}~$(K,\cubes{K})$ is a subspace~$K \subsetneq \R^n$ that is equal to
the  union of a finite set of elementary cubes, together with the stratification
$\cubes{K}$ induced
by the cubical stratification of~$\R^n$; see \figref{ecc}. 

 We topologize $K$
using the subspace topology with the standard topology on $\R^n$.
By construction, if $\sigma \in \cubes{K}$, then all of its faces are
necessarily in~$\cubes{K}$ as well.  If~$\sigma \in \cubes{K}$ with no proper
cofaces, then we say that $\sigma$ is a \emph{maximal cube} in~$K$.
We denote the set of closed cubes in $(K,\cubes{K})$ by $\clcubes{K}$; the set
of closed cubes in $\clcubes{K}$ is in one-to-one correspondence with the open
cubes in $\cubes{K}$. Specifically, vertices in~$\clcubes{K}$ correspond to
vertices in~$\cubes{K}$ and all other elementary cubes in~$\clcubes{K}$ correspond to their interiors in
$\cubes{K}$. Throughout this paper, we denote the set of zero-cubes in $\cubes{K}$ by
$\verts{K}$ and note that~$\verts{K} \subsetneq \Z^n$, since all cubes in
$(K,\cubes{K})$ are
elementary cubes.

The \emph{product order on $\R^n$}, denoted $\preceq$, is the partial order
such that for two points~$\vecp = (p_1,p_2, \ldots, p_n)$ and $\vecq=(q_1,q_2, \ldots, q_n)$
in~$\R^n$, we have that $\vecp \preceq \vecq$ if and only if $p_i\leq q_i$ for
each coordinate~$i$. Using this partial order, we define the interval of points
in $\R^n$ between $\vecp$ and $\vecq$ as
    $$[\vecp,\vecq]
    := \{\vecx \mid \vecp\preceq \vecx
    \preceq \vecq \}.$$
The point $\vecp$ is the minimum vertex of the interval
and $\vecq$ is the maximum vertex
of the interval, with respect to $\preceq$.  Notationally, we write this as~$\min([\vecp,\vecq]):=\vecp$
and~$\max([\vecp,\vecq]):=\vecq$.
When $\vecq \in \Z^n$ and~$\vecq=\vecp+\vecj$, for some~$\vecj \in \{0,1\}^n$,
the interval~$[\vecp,\vecq]$
is an elementary cube as defined in \eqnref{elem-cube}.
If, in addition, $\vecj$ is not the zero vector,
then we say that~$[\vecv-\vecj,\vecv]$ is a \emph{lower coface} of $\vecv$.

Using the fact that the partial order $(\R^n,\preceq)$ induces a partial order on the
points in $K$, we define directed paths in $K$ as the set of
nondecreasing paths in $K$:
A \emph{path} in $K$ is a continuous map from $I=[0,1]$ to~$K$, where $I$ is the unit interval. We say that a path $\gamma \colon I \to K$ goes
from $\gamma(0)$ to~$\gamma(1)$.
Letting $K^I$ denote the set of all paths in $K$,
the set of \emph{directed paths} (or \emph{dipaths} for short) is
$$
    \di{K} := \{ \gamma \in K^I ~|~ \forall i, j \text{ s.t.\ }
    0 \leq i \leq j \leq 1, \gamma(i) \preceq \gamma(j) \}.
$$
We topologize $\di{K}$ using the compact-open topology.  For $\vecp,\vecq \in K$, we
denote the subspace of dipaths from~$\vecp$ to $\vecq$
by~$\dipaths{K}{\vecp}{\vecq}$.
We refer to $(K, \di{K})$ as a \emph{directed Euclidean cubical
complex}.\footnote{Directed Euclidean cubical complexes are an example of a more general
concept known as \emph{directed space} (d-spaces).  To define a d-space, we have a
topological space $X$ and we define a set
of dipaths $P'(X) \subseteq X^I$ that  contains all constant paths,
and is closed under taking nondecreasing
reparameterizations, concatenations, and
subpaths. Indeed, $\di{K}$ satisfies these properties.}
The connected components of $\dipaths{K}{\vecp}{\vecq}$ are exactly the
equivalence classes of dipaths, up to dihomotopy.
If two dipaths,~$f$ and~$g$ are homotopic through a continuous family
of dipaths, then~$f$ and~$g$ are called~\emph{dihomotopic}.

Given a directed complex, certain subcomplexes are of~interest:
\begin{definition}[Special Complexes]\label{def:subcomplexes}
    Let $(K,\cubes{K})$ be a directed Euclidean cubical complex in $\R^n$.
    Let $\vecp \in \verts{K}$ and let $\sigma$ be an elementary cube (that need not be in
    $\cubes{K}$).
    \begin{enumerate}
        \item The complex above $\vecp$ is $\upset{K}{\vecp} := \{
                \vecq \in K \mid \vecp \preceq \vecq\}$.
        \item The complex below $\vecp$ is $\downset{K}{\vecp} := \{
                \vecq \in K \mid \vecq \preceq \vecp\}$.
        \item The reachable complex from $\vecp$ is $\reachcplx{K}{\vecp} := \{
                \vecq \in K \mid \dipaths{K}{\vecp}{\vecq}
                \neq \emptyset \}$.\label{stmt:subcomplexes-reachable}
        \item The complex restricted to $\sigma$ is
            $$K|_{\sigma} := \bigcup \{ \tau \in \cubes{K}
                \mid \min{\sigma} \preceq \min{\tau} \preceq
                    \max{\tau} \preceq \max{\sigma} \}.$$
        \item If $K=I^n$, then we call $(K, \cubes{K})$ the
                    \emph{standard unit cubical complex} and often denote it by
                    $(I^n, \cubes{I})$.  If $K = I^n + \vecx$ for some $\vecx
                    \in \Z^n$, then $K$ is a full-dimensional unit cubical complex.
    \end{enumerate}
\end{definition}

\subsection{Past Links of Directed Cubical Compelxes}
An \emph{abstract simplicial complex} is a finite collection $\mathcal{S}$ of sets that is closed under
the subset relation, i.e., if $A\in \mathcal{S}$ and $B$ is a set such that $\emptyset
\neq B \subseteq A$, then~$B\in \mathcal{S}$. The sets in $\mathcal{S}$ are called \emph{simplices}.
If the simplex~$A$ has $k+1$ elements, then we say that
the dimension of~$A$ is $\dim(A):=k$, and we say that $A$ is a $k$-simplex.
For example, the zero-simplices are the singleton sets and are often referred to
as vertices. Since every element of a set $A$ in $\mathcal{S}$ gives rise to a singleton set in the finite set $\mathcal{S}$, $A$ must be finite.

In a topological space embedded in $\R^n$, the
link of a point $\vecv$ is constructed by intersecting an arbitrarily small
$(n-1)$-sphere around $\vecv$ with the space itself.
In $\R^n$, the link of a point is
an~$(n-1)$-sphere. Moreover, if ~$\vecv \in \Z^n$, the link inherits the stratification as a subcomplex
of~$\R^n$, and can be represented as a simplicial complex whose~$i$-simplices are
in one-to-one correspondence with the~$(i+1)$-dimensional cofaces of $\vecv$.
The past link of $\vecv$ is the restriction of the link
using the set of lower cofaces of $\vecv$ instead of all cofaces.
Thus, we can represent each simplex in the past link as a
vector in~$\{0,1\}^n \setminus \setzero$, where the vector $\vecj \in \{0,1\}^n
\setminus \setzero$
represents the cube~$[\vecv-\vecj,\vecv]$ in the simplex-cube
correspondence.
As a simplicial complex, the past link of $\vecv$ in~$\R^n$ has~$n$
vertices~$\{ x_i \}_{1 \leq i \leq n}$, and~$\vecj$
represents the simplex~$\{ x_i | 1 \leq i \leq n, j_i=1\}$ of dimension
$||\vecj||_1-1$; for example, $(1,0,0)$ represents a vertex and $(1,0,1)$ represents an edge.
We are now ready to define the past link of a vertex in a Euclidean cubical
complex:
\begin{definition}[Past Link]\label{def:pl}
    Let $(K,\cubes{K})$ be a directed Euclidean cubical complex in $\R^n$.
    Let $\vecv \in \Z^n$.
    The
    \emph{past link} of $\vecv$ is the following simplicial complex:
    $$
    \pastlk{K}{\vecv} := \{
        \vecj \in \{0,1\}^n \setminus \setzero  \mid
        [\vecv-\vecj,\vecv] \subseteq K
    \}.$$
\end{definition}
As a set, the past link represents all
elementary cubes in $K$ for which the maximum vertex is $\vecv$.
As a simplicial complex, it describes (locally) the different types of dipaths
to or through $\vecv$ in $K$; see \figref{box-pastlink}.
\begin{figure}%
    \centering
    \subfloat[Past Link of $\vecv$\label{subfig:box-pastlink-v}]{%
        \includegraphics[height=1in]{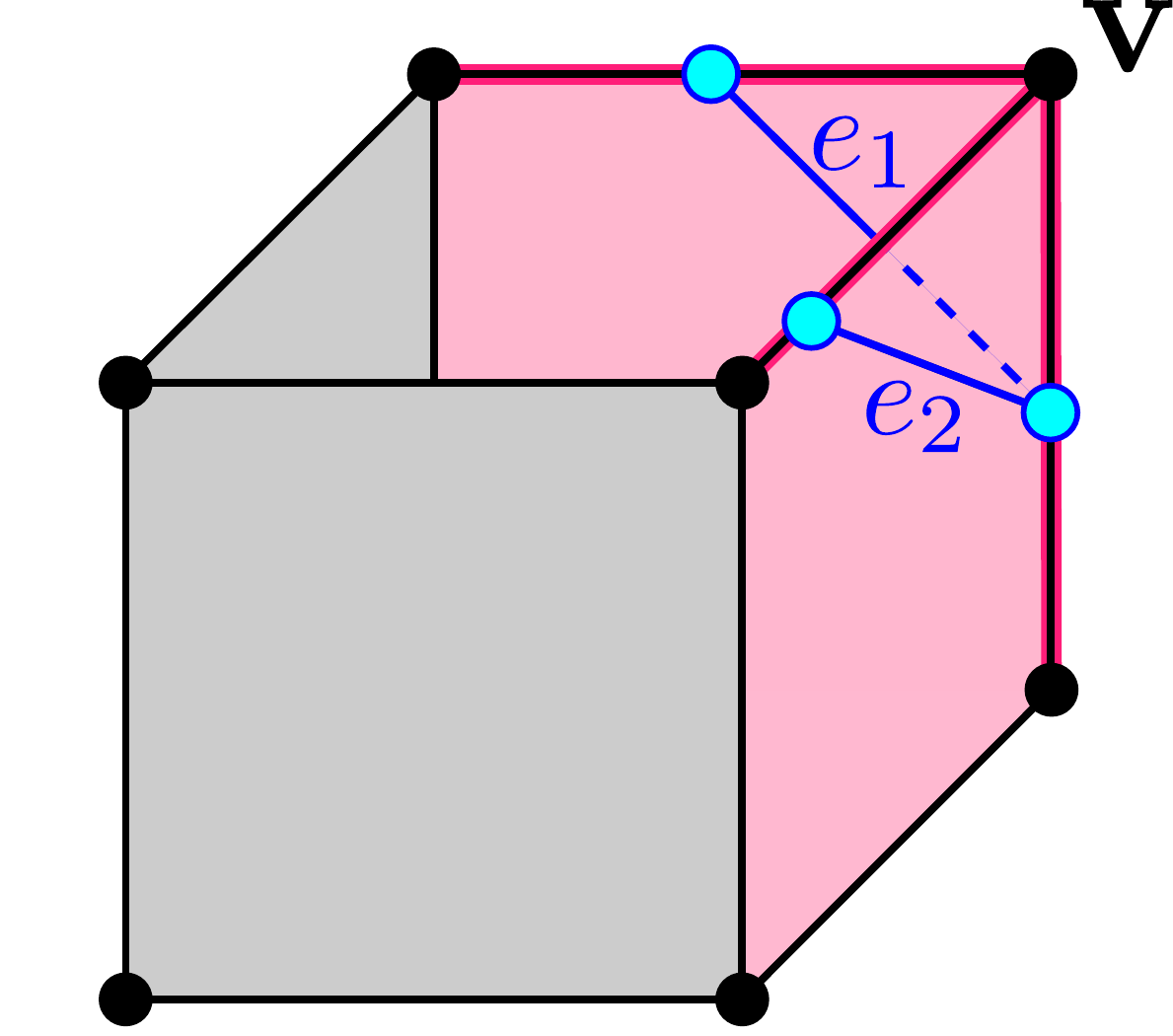}%
    }\hfil
    \subfloat[Past Link of $\vecx$\label{subfig:box-pastlink-x}]{%
        \includegraphics[height=1in]{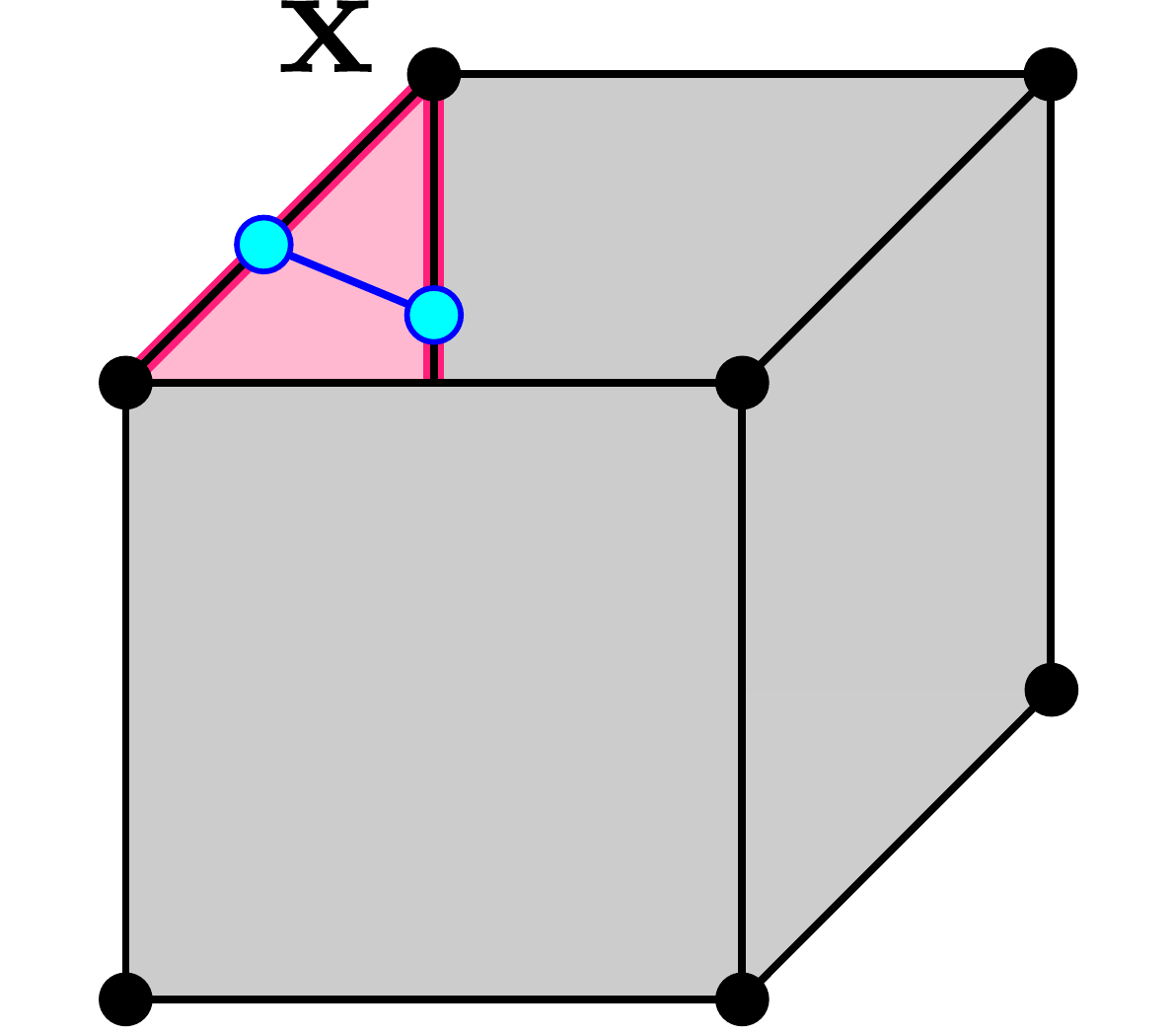}%
    }
    \caption{Past link in the Open Top Box.
    \protect\subref{subfig:box-pastlink-v}
    The maximum vertex of this complex is $\vecv = (v_1, v_2, v_3)$.
    The past link $\pastlk{K}{\vecv}$ is the simplicial complex
    comprising three vertices and two edges
    (shown in blue/cyan).
    These simplices are in one-to-one correspondence with the set of lower cofaces
    of~$\vecv$ (highlighted in
    pink).
    For example, the edges of
    $\pastlk{K}{\vecv}$, which are labeled~$e_1$ and~$e_2$,
    are in one-to-one correspondence with the
    elementary two-cubes that are lower cofaces of $\vecv$
    ($\sigma_1= [(v_1-1,v_2,v_3-1),\vecv]$
    and~$\sigma_2=[(v_1,v_2-1,v_3-1),\vecv]$, respectively).
    In the vector notation for simplices of~$\pastlk{K}{\vecv}$,
    we write~$e_1=(1,0,1)$ and $e_2=(0,1,1)$.
    \protect\subref{subfig:box-pastlink-x}
    The past link of a vertex $\vecx$ that is neither the minimum nor the
    maximum vertex in the~complex.
    \label{fig:box-pastlink}}
\end{figure}

We conclude
this section with a lemma summarizing properties of the past link, most of which
follow directly from definitions:

\begin{lemma}[Properties of Past Links]\label{lem:plprops}
    Let $(K,\cubes{K})$ be a directed Euclidean cubical complex in $\R^n$.
    Then, the following statements hold for all $\vecv \in \Z^n$:
    \begin{enumerate}
        \item $\pastlk{K}{\vecv} = \bigcup_{\vecp \in \R^n}
            \pastlk{\upset{K}{\vecp}}{\vecv}$.\label{stmt:plprops-bigcup}
        \item If $(K', \cubes{K'})$ is a subcomplex of $(K, \cubes{K})$, then
            $\pastlk{K'}{\vecv} \subseteq
            \pastlk{K}{\vecv}$.\label{stmt:plprops-subcomplex}
        \item $\pastlk{K}{\vecv} =
            \pastlk{\downset{K}{\vecv}}{\vecv}$.\label{stmt:plprops-upset}
        \item If there exists
            $\vecw \in \Z^n$ such that $K=[\vecw-\binone,\vecw]$, then
            $\pastlk{K}{\vecw}$ is the complete simplicial complex on $n$
            vertices.\label{stmt:plprops-full}
        \item $\pastlk{K}{\vecv}$ is a subcomplex of the complete simplicial
            complex on $n$ vertices.\label{stmt:plprops-subsimplex}
  
    \end{enumerate}
\end{lemma}

\begin{proof}
    \stmtref{plprops-bigcup}:
    If $K = \emptyset$, then all past links are empty and the equality trivially
    holds.
    If $K \neq \emptyset$, then $\verts{K}$ is a finite non empty set.  Thus, there
    exists~$\vecq \in \R^n$ such that for all $\vecw \in \verts{K}$,  $\vecq
    \preceq \vecw$.
    
    Let $\vecj \in \pastlk{K}{\vecv}$.  Then,~$[\vecv-\vecj,\vecv] \subseteq {K}$
    and so $\vecv-\vecj \in \verts{K}$.  Hence,~$\vecq \preceq
    \vecv-\vecj$, which means that~$\vecj \in \pastlk{\upset{K}{\vecq}}{\vecv} \subseteq
    \bigcup_{\vecp \in \R^n} \pastlk{\upset{K}{\vecp}}{\vecv}$.  The reverse
    inclusion follows from the fact that each of these statements holds if and
    only~if.

    \stmtref{plprops-subcomplex}: Observe that if $\vecj \in
    \pastlk{K'}{\vecv} $, then, by definition of the
    past link,~$[\vecv-\vecj,\vecv] \subseteq {K'}$.
    Since $K'\subseteq K$,
    we have $[\vecv-\vecj,\vecv] \subseteq {K'}\subseteq {K}$.
    Therefore, we can conclude that $\vecj  \in \pastlk{K}{\vecv}$.

    \stmtref{plprops-upset}: By \stmtref{plprops-subcomplex} (which we just
    proved),
    we have the following
    inclusion $\pastlk{\downset{K}{\vecv}}{\vecv}\subseteq \pastlk{K}{\vecv}$.
    To prove the inclusion in the other direction, let
    $\vecj \in \pastlk{K}{\vecv}$. Since~$\vecv-\vecj \preceq \vecv$, then
    $[\vecv-\vecj,\vecv] \subseteq {\downset{K}{\vecv}}$. Therefore, we
    conclude that~$\pastlk{K}{\vecv} \subseteq
    \pastlk{\downset{K}{\vecv}}{\vecv}$.

    \stmtref{plprops-full}: Since
    $K=[\vecw-\binone,\vecw]$, we know that
    $K$ is full-dimensional, and so for all $\vecj \in
    \{0,1\}^n$, $[\vecw-\vecj,\vecw] \subseteq K$.  Thus, by
    definition of past link, we have that the past link of $\vecw$ is:
    $\pastlk{K}{\vecw} := \{0,1\}^n \setminus \setzero$, which is the complete
    simplicial complex on $n$ vertices.

    \stmtref{plprops-subsimplex}:
    Let $L=K\cap[\vecv-\binone,\vecv]$.
    By definition of past link, we know~$\pastlk{L}{\vecv} = \pastlk{K}{\vecv}$.
    By \stmtref{plprops-subcomplex}, since
    $L$ is a subcomplex of the cube $[\vecv-\binone,\vecv]$,
    we know~$\pastlk{L}{\vecv} \subseteq \pastlk{[\vecv-\binone,\vecv]}{\vecv}$.
    By \stmtref{plprops-full}, the past link $\pastlk{[\vecv-\binone,\vecv]}{\vecv}$ is the complete simplicial
    complex on $n$ vertices.
    Therefore, $\pastlk{K}{\vecw}$ is the complete simplicial complex on $n$~vertices.
\end{proof}

\subsection{Relationship Between Past Links and Path Spaces}\label{ssec:prior}
The topology of the past links is intrinsically related to spaces
of dipaths. Specifically, in~\cite{belton2020towards} we prove that
the contractability and/or connectedness of past links of vertices in directed
Euclidean cubical complexes
with a minimum vertex\footnote{In~\cite{belton2020towards}, the minimum
(initial) vertex
was often assumed to be $\zero$ for ease of exposition.  We restate the lemmas and
theorems here using more general notation, where $K$ has a minimum vertex~$\vecw$.}
implies that all spaces of dipaths with $\vecw$ as initial point are also contractible and/or connected.

\begin{theorem}[{Contractability~\cite[Theorem
    1]{belton2020towards}}]\label{thm:partial-contractability}
    Let $(K,\cubes{K})$ be a directed Euclidean cubical complex in $\R^n$
    that has a minimum vertex $\vecw$.
    If, for all vertices~$\vecv \in \verts{K}$,
    the past link $\pastlk{K}{\vecv}$  is contractible, then
    the space~$\dipaths{K}{\vecw}{\veck}$ is contractible for all $\veck \in
    \verts{K}$.
\end{theorem}

An analogous theorem for connectedness also holds.

\begin{theorem}[{Connectedness~\cite[Theorem 2]{belton2020towards}}]\label{thm:connected}
    Let $(K,\cubes{K})$ be a directed Euclidean cubical complex in $\R^n$
    that has a minimum vertex $\vecw$.
    Suppose that, for all $\vecv \in \verts{K}$, the past link $\pastlk{K}{\vecv}$ is connected. Then,
    for all $\veck \in \verts{K}$,
    the space~$\dipaths{K}{\vecw}{\veck}$ is connected.

\end{theorem}

Furthermore, we proved a partial converse to \thmref{connected}. Specifically,
the converse holds only if $K$ is a reachable directed Euclidean cubical complex as defined
in~\defstmtref{subcomplexes}{reachable}. This is expected: properties of  parts of the directed
Euclidean complex which are not reachable from the vertex $\vecw$, do not influence the
dipath spaces from $\vecw$.

\begin{theorem}[{Realizing Obstructions \cite[Theorem 3]{belton2020towards}}]\label{thm:obstructions}
    Let $(K,\cubes{K})$ be a directed Euclidean cubical complex in $\R^n$.
    Let $\vecw \in \verts{K}$, and
    let $L = \reachcplx{K}{\vecw}$.
    Let~$\vecv \in \verts{L}$.
    If the past link $\pastlk{L}{\vecv}$ is disconnected, then
    the space~$\dipaths{K}{\vecw}{\vecv}$ is disconnected.
\end{theorem}

\section{Directed Collapsing Pairs}
Although simplicial collapses preserve the homotopy type of the underlying
space \cite[Proposition 6.14]{kozlovcombinatorial2007} and hence of all path spaces, this type of collapsing in directed Euclidean
cubical complexes may not preserve topological properties of spaces of dipaths. In this section,
we study a specific type of collapsing called a link-preserving directed collapse.
We define link-preserving directed collapses in \secref{LPDC-def} and give properties of link-preserving
directed collapses in \secref{LPDC-properties}.

\subsection{Link-Preserving Directed Collapses} \label{sec:LPDC-def}
Since we are interested in preserving the dipath spaces through
collapses, the results from \ssecref{prior}
motivate us to study a type of directed collapse (DC) via past links, introduced in \cite{belton2020towards}.
However, we call it a \emph{link-preserving directed collapse} (LPDC) (as
opposed to a \emph{directed collapse})
since we show in the last sections of this paper that when the spaces of dipaths starting
from the minimum vertex are not connected, the following definition of collapse does not preserve
the number of components.

\begin{definition}[Link Preserving Directed Collapse]\label{def:directedCollapse}
    Let $(K,\cubes{K})$ be a directed Euclidean cubical complex in $\R^n$.
    Let $\sigma \in \cubes{K}$ be a maximal cube, and let $\tau$ be a proper
    face of $\sigma$ such that no
    other maximal cube contains the face $\tau$ (in this case, we say that $\tau$ is a
    \emph{free face} of $\sigma$). Then, we define the $(\tau,\sigma)$-collapse
    of $K$ as the subcomplex obtained by
    removing everything in between $\tau$ and $\sigma$:
    \begin{equation}\label{eqn:directedCollapse}
    K' =
        K \setminus \{\gamma \in \cubes{K}  \mid
            \overline{\tau} \subseteq \overline{\gamma} \subseteq
            \overline{\sigma} \},
    \end{equation}
    and let $\cubes{K'}$ denote the stratification of the set $K'$
    induced by the cubical stratification of $\R^n$ (thus, $\cubes{K'} \subsetneq
    \cubes{K}$).

    We call the directed Euclidean cubical complex~$(K',\cubes{K'})$
    a \emph{link-preserving directed collapse (LPDC)} of $(K,\cubes{K})$  if,
    for all~$\vecv\in \verts{K'}$,
    the past link~$\pastlk{K}{\vecv}$ is homotopy equivalent to
    $\pastlk{K'}{\vecv}$ (denoted $\pastlk{K}{\vecv} \simeq \pastlk{K'}{\vecv}$).
    The pair $(\tau,\sigma)$ is then called an \emph{LPDC pair}.
\end{definition}

\begin{remark}[Simplicial Collapses]
    The study of simplicial collapses
    is known as \emph{simple homotopy
    theory}~\cite{whitehead1950simple,cohen2012course}, and traces back to the work of
    Whitehead in the 1930s~\cite{whitehead}.  The idea is very similar:
    If $C$ is an abstract simplicial complex and $\alpha \in C$ such that $\alpha$ is a
    proper face of exactly one maximal simplex $\beta$, then the following
    complex is the \emph{$\alpha$-collapse of $C$ in the simplicial setting}:
    $$C'=C \setminus \{
    \gamma \in C \mid \alpha \subseteq \gamma \subseteq \beta \}.$$
    Note that we use only the free face ($\alpha$) when
    defining a simplicial collapse, as
    doing so helps to distinguish between discussing
    a simplicial collapse and
    a directed Euclidean cubical collapse.
    In addition, we always explicitly state ``in the simplicial setting" when
    talking about a simplicial collapse.
\end{remark}

Applying a sequence of LPDCs to a directed Euclidean cubical complex can reduce the number
of cubes, and hence can more clearly illustrate the number of
dihomotopy classes of dipaths within the directed Euclidean cubical complex. For an example, see
\figref{swissflag-collapse}. However, it is not necessarily true that LPDCs preserve dipath spaces.
We discuss the relationship between dipath spaces and LPDCs in \secref{pathspaces}.

\begin{figure}[th]
    \centering
    \subfloat[Swiss Flag\label{subfig:swiss-collapse1}]{%
        \includegraphics[height=1in]{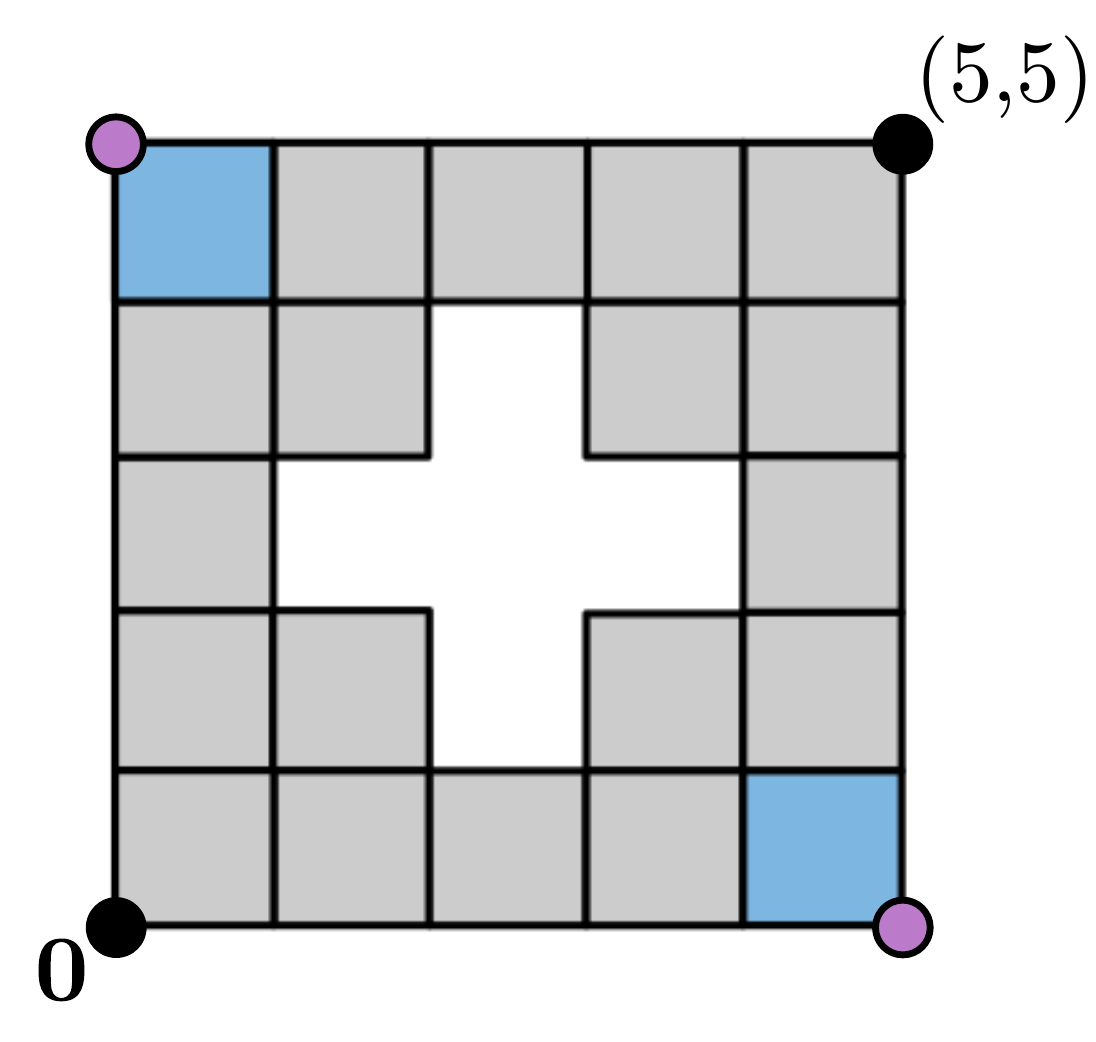}}%
    \hfil
    \subfloat[First stage\label{subfig:swiss-collapse2}]{%
        \includegraphics[height=1in]{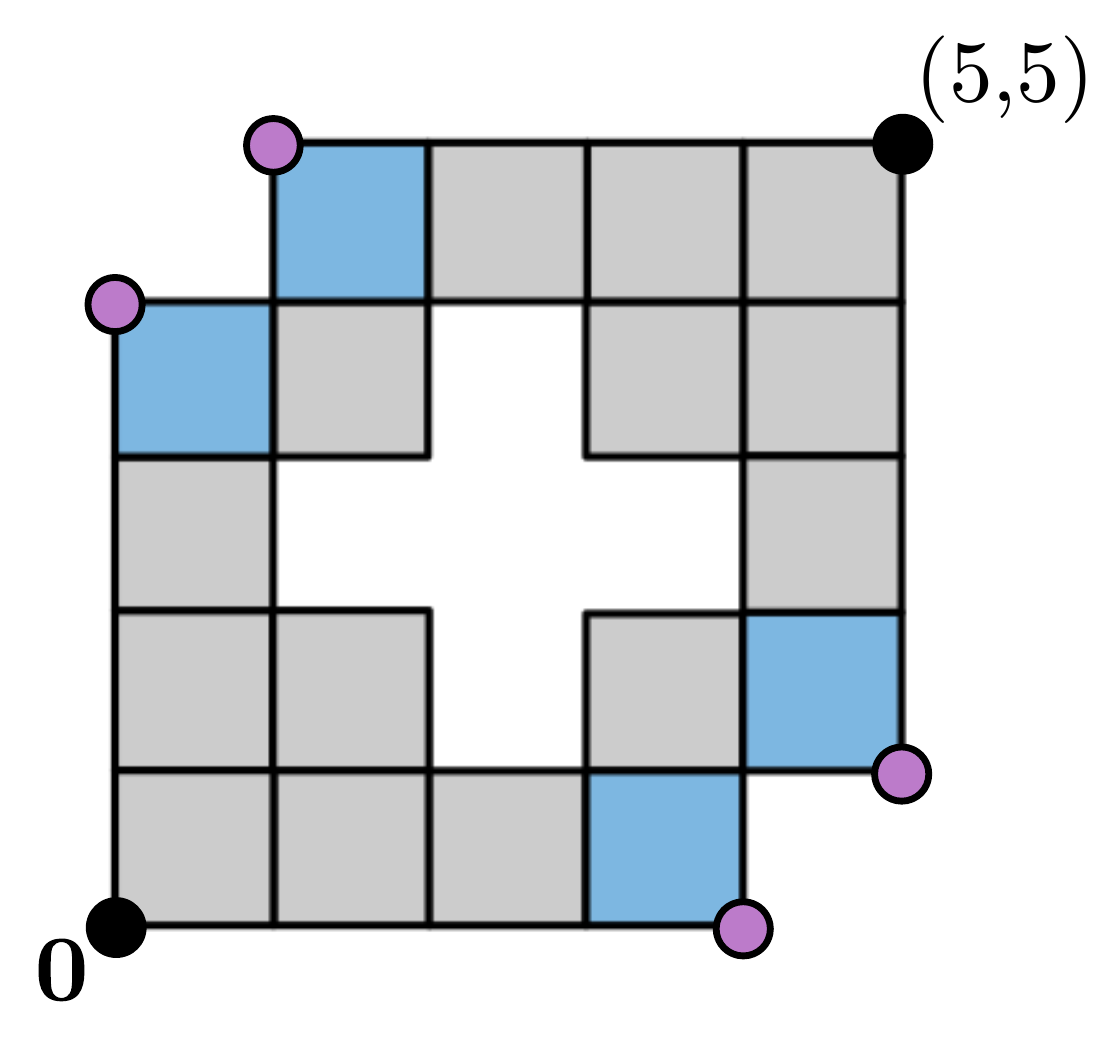}}%
    \hfil
    \subfloat[Second stage\label{subfig:swiss-collapse3}]{%
        \includegraphics[height=1in]{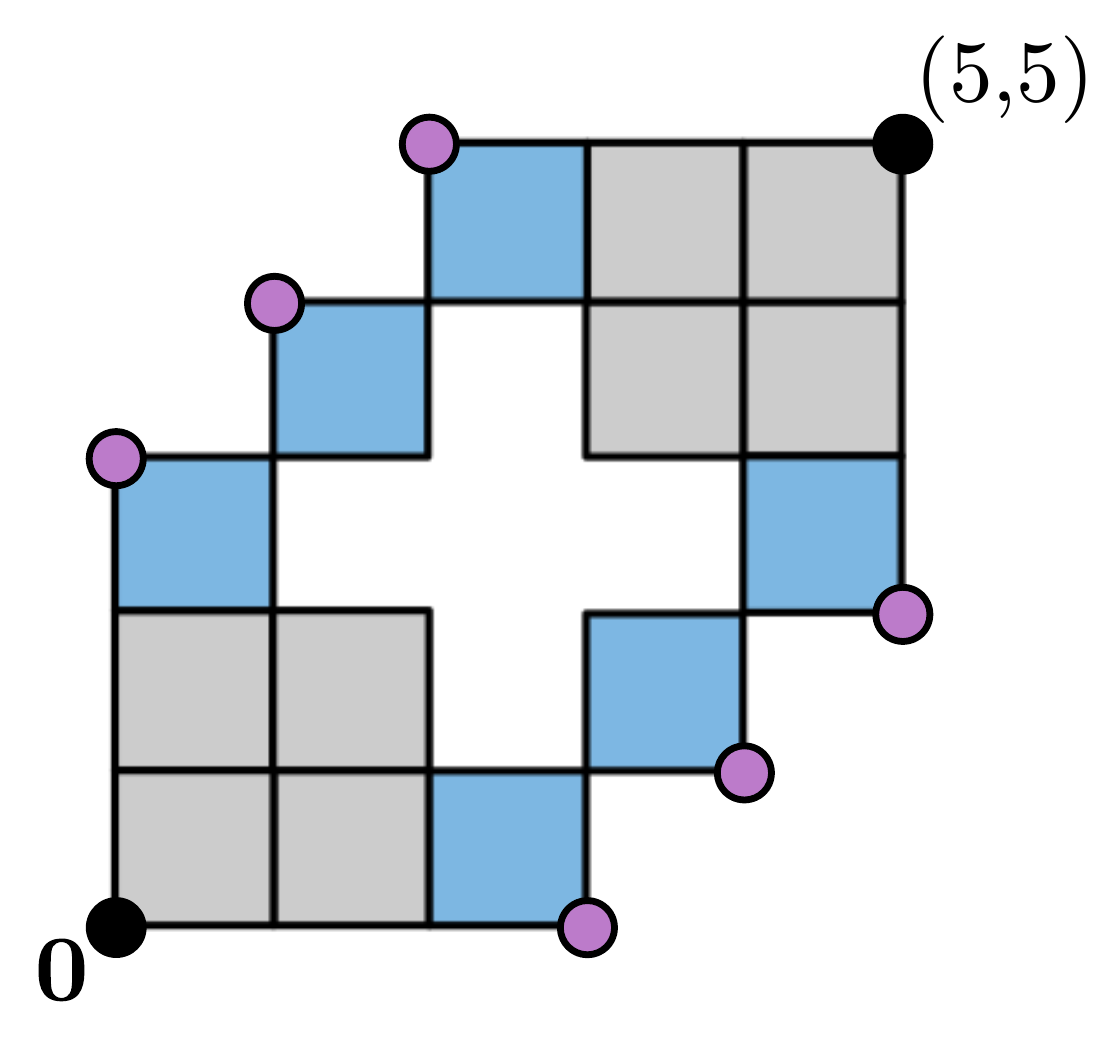}}\\
    \subfloat[Third stage\label{subfig:swiss-collapse4}]{%
        \includegraphics[height=1in]{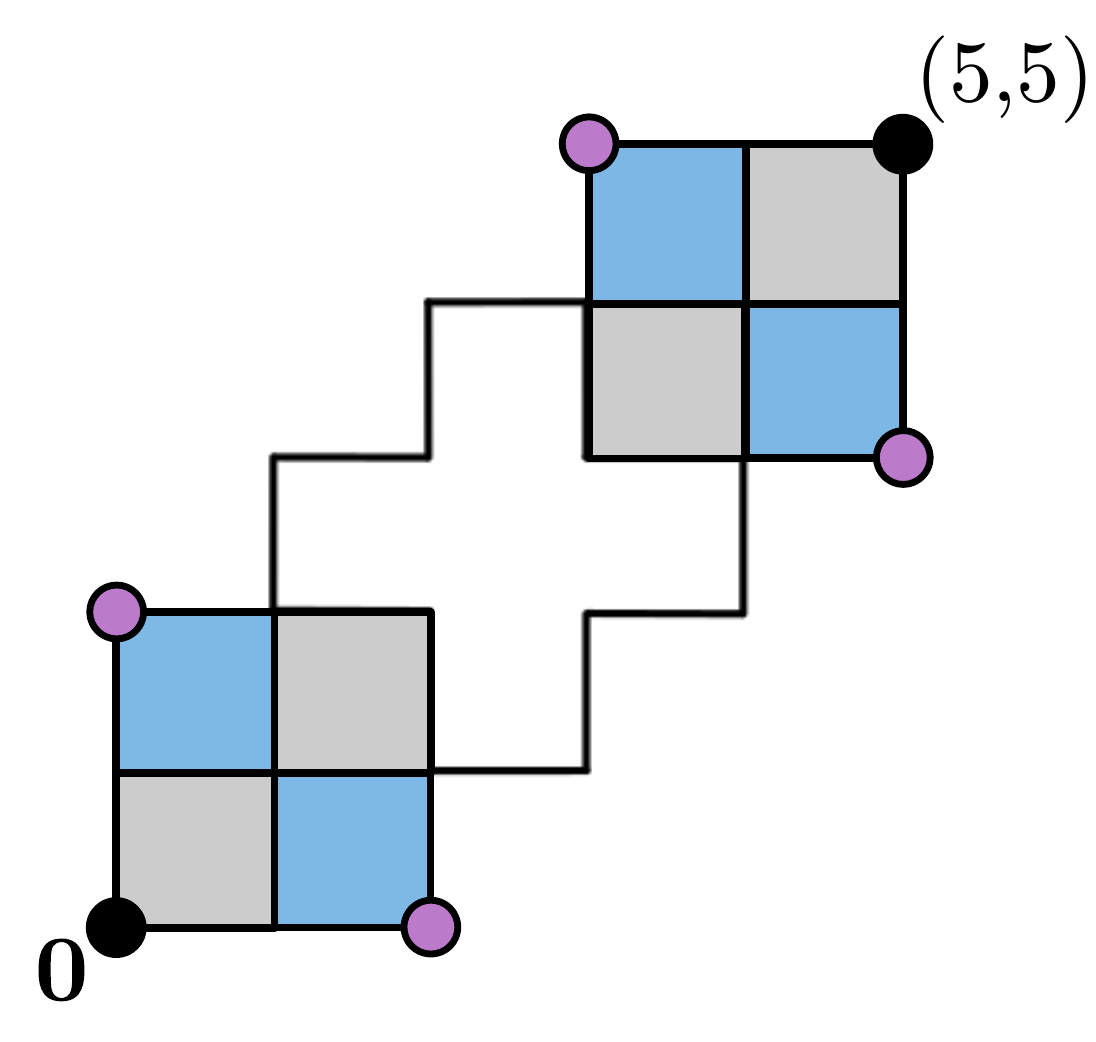}}%
    \hfil
    \subfloat[Fourth stage\label{subfig:swiss-collapse5}]{%
        \includegraphics[height=1in]{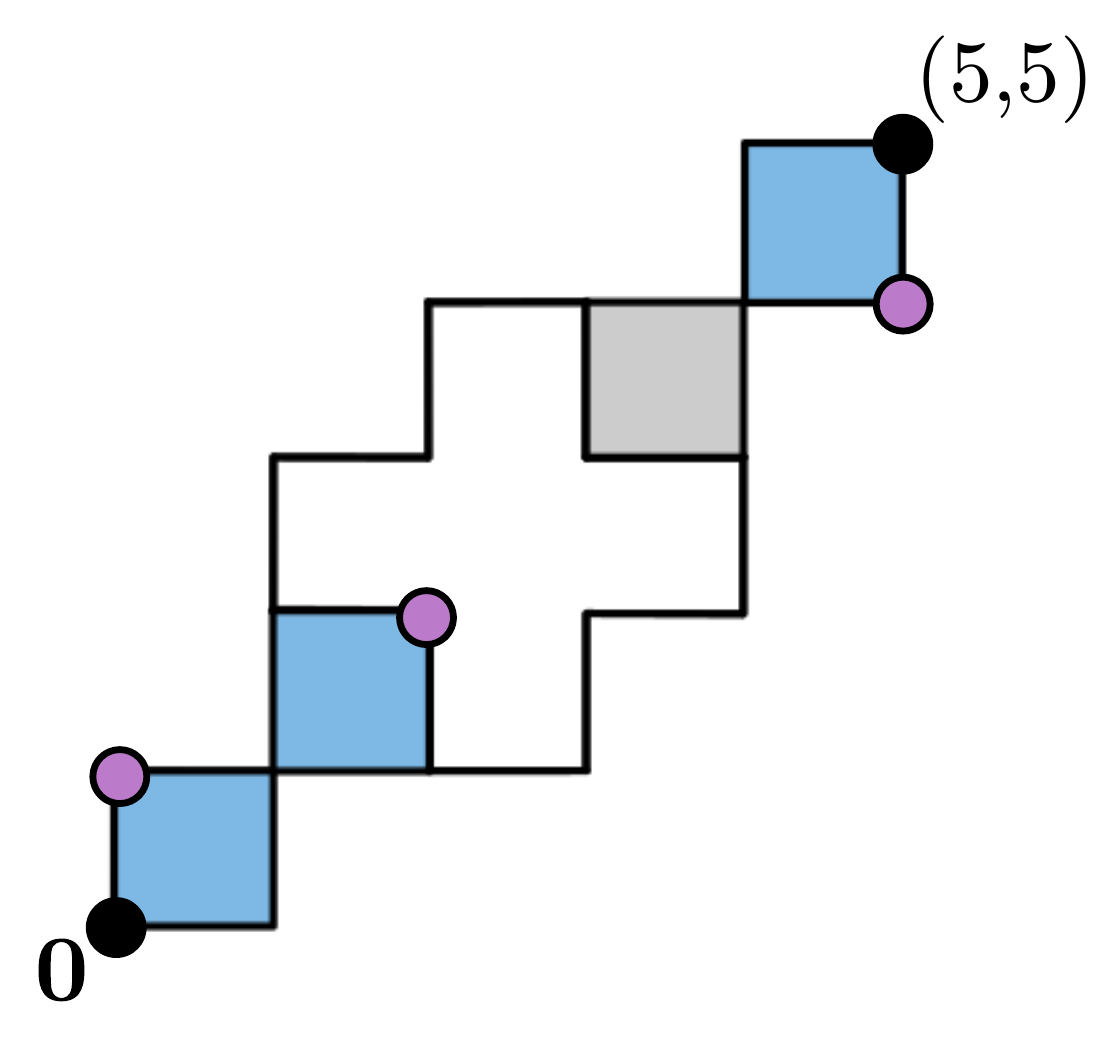}}%
    \hfil
    \subfloat[After Collapses\label{subfig:swiss-collapse-final}]{%
        \includegraphics[height=1in]{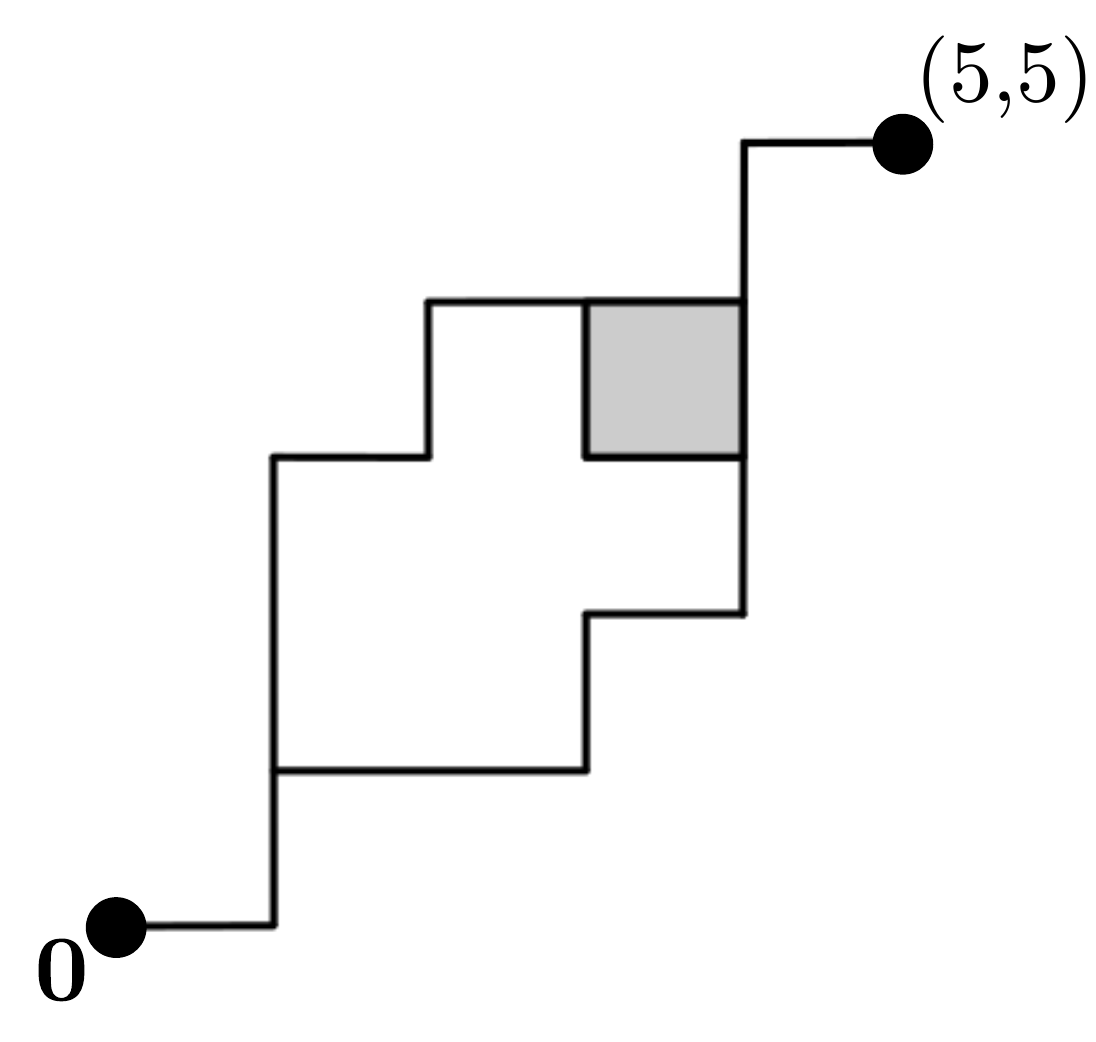}}%
    \caption{Collapsing the Swiss Flag. A sequence of vertex collapses is
    presented from the top left to bottom right.  At each stage of the sequence, the faces and
    vertices shaded in blue and purple represent the vertex collapsing
    pairs with the blue Euclidean cube
    being $\sigma$ and the purple vertex being $\tau$. The
    result of the sequence of LPDCs is shown in \protect\subref{subfig:swiss-collapse-final} is a
    one-dimensional directed Euclidean cubical complex and one two-cube. Observe that this
    directed Euclidean cubical complex clearly illustrates the two dihomotopy classes of $\protect\dipaths{K}{\zero}{(5,5)}$.
    }
    \label{fig:swissflag-collapse}
\end{figure}

\subsection{Properties of LPDCs} \label{sec:LPDC-properties}

We give a combinatorial condition for a collapsing pair  $(\tau, \sigma)$ to be
an LPDC pair; namely, the condition is that $\tau$ does not contain the minimal vertex of $\sigma$,
$\min(\sigma)$. From the definition of an LPDC, we see that finding an LPDC pair requires
computing the past link of \emph{all} vertices in $\verts{K'}$. In \cite{belton2020towards}, we
discuss how we can reduce to only checking the vertices in ~$\sigma$, since no other
vertices have their past links affected. In this paper, we prove we need to only
check \emph{one} condition to determine if we have an LPDC pair. The one simple
condition dramatically reduces the number of computations we need to perform in order to verify
we have an LPDC. This result  given in~\thmref{collapsingpairs} depends on the following lemmas
about the properties of past links on vertices.

\begin{lemma}[Properties of Past Links in a Vertex Collapse]\label{lem:pastlink}
    Let $(K,\cubes{K})$ be a directed Euclidean cubical complex in $\R^n$.
    Let~$\sigma \in \cubes{K}$ and~$\tau,\vecv \in \verts{\sigma}$
    such that $\tau \preceq \vecv$. If $\tau$ is a free face of $\sigma$ and~$K'$ is
    the~$(\tau,\sigma)$-collapse,
    then the following two statements hold:
    \begin{enumerate}
        \item $\pastlk{K|_{\sigma}}{\vecv}=\{\vecj \in\{0,1\}^n\setminus\setzero \mid
            \min(\sigma) \preceq \vecv - \vecj \}$.\label{stmt:pastlink-K}
        \item
            $\pastlk{K'|_{\sigma}}{\vecv}=\pastlk{K|_{\sigma}}{\vecv}\setminus\{\vecj\in\{0,1\}^n
            \setminus \setzero
            \mid \vecv- \vecj \preceq \tau \}$.\label{stmt:pastlink-Kprime}
    \end{enumerate}
\end{lemma}

\begin{proof}
    To ease notation, we define the following two sets:
    \begin{align*}
        J & := \{\vecj\in\{0,1\}^n\setminus\setzero \mid \min(\sigma) \preceq \vecv - \vecj \} \\
        I & := \{\vecj\in\{0,1\}^n\setminus\setzero \mid \vecv- \vecj \preceq \tau \}.
    \end{align*}

    First, we prove \stmtref{pastlink-K} (that $\pastlk{K|_{\sigma}}{\vecv} = J$).
    We start with the forward inclusion. Let~$\vecj\in \pastlk{K|_{\sigma}}{\vecv}$.
    By the definition of past links (see \defref{pl}), we know that
    $[\vecv-\vecj,\vecv]\subseteq K|_{\sigma}$.
    By the definition of $K|_{\sigma}$ (see \defref{subcomplexes}),
    we know that~$\min(\sigma) \preceq \min([\vecv-\vecj,\vecv]) = \vecv-\vecj$. This then implies $\vecj\in J$.
    Therefore,~$\pastlk{K|_{\sigma}}{\vecv}\subseteq J$.
    For the backward inclusion, let~$\vecj\in J$.
    Then, since $\vecv \in \verts{\sigma}$ and
    $\sigma$ is an elementary cube by assumption,
    and $\min(\sigma) \preceq \vecv -
    \vecj$ by definition of~$J$, we have that $\vecv-\vecj$ is an element of $\verts{\sigma}$. Since~$\sigma \in \cubes{K}$, all faces
   of $\sigma$ must be in $\cubes{K}$; hence,~$[\vecv-\vecj,\vecv]\subseteq K|_{\sigma}$. Therefore,~$\vecj \in
    \pastlk{K|_{\sigma}}{\vecv}$,
    and so~$\pastlk{K|_{\sigma}}{\vecv} \supseteq J$.
    Since we have both inclusions, then \stmtref{pastlink-K} holds.

    Now, we prove \stmtref{pastlink-Kprime} (that $\pastlk{K'|_{\sigma}}{\vecv}
    = J \setminus I$).
    Again, we prove the inclusions in both
    directions, starting with the forward inclusion. Consider $\vecj \in
    \pastlk{K'|_{\sigma}}{\vecv}$.
    By \lemstmtref{plprops}{subcomplex}, we have the following inclusion of past links: 
    $\pastlk{K'|_{\sigma}}{\vecv} \subseteq \pastlk{K|_{\sigma}}{\vecv}$,
    and so, we obtain~$\vecj \in \pastlk{K|_{\sigma}}{\vecv} = J$.
    Next, we must show that $\vecj \notin I$.
    Assume, for a contradiction, that $\vecj \in I$. Then, by definition of $I$,
    $\vecv-\vecj \preceq \tau$. Since~$\tau \preceq \vecv$,  we
    obtain the partial order~$\vecv -\vecj \preceq \tau \preceq \vecv$.
    This implies that $[\tau, \vecv]$ must be a face of  $[\vecv-\vecj,\vecv]$, i.e. $[\tau, \vecv]\subseteq [\vecv-\vecj,\vecv]$.
    Since~$[\vecv-\vecj,\vecv]$ is an elementary cube in $K' |_{\sigma}$, then
    its face $[\tau,
    \vecv]$ must also be an elementary cube in $K' |_{\sigma}$.
    Setting $\overline{\gamma} = [\tau,\vecv]$ and observing $\tau =
    \overline{\tau} \subseteq \overline{\gamma}
    \subseteq \overline{\sigma}$, we observe that~$\gamma$ is not an elementary
    cube in $K'$ by \eqnref{directedCollapse}.
    This gives us a contradiction and so~$\vecj \notin I$.
    Therefore, $\pastlk{K'|_{\sigma}}{\vecv} \subseteq J \setminus I$.

    Finally, we prove the backward inclusion of \stmtref{pastlink-Kprime}. Let
    $\vecj \in J \setminus I$.
    Then, by \stmtref{pastlink-K}, $\vecj \in \pastlk{K|_{\sigma}}{\vecv}$ and either $\tau
    \prec \vecv-\vecj$ or $\tau$
    is not comparable to~$\vecv-\vecj$ under $\preceq$.
    Thus, by \eqnref{directedCollapse}, $[\vecv-\vecj,\vecv]$
    is an elementary cube of~$K'|_{\sigma}$.
    Thus, by \defref{pl}, we have that $\vecj \in \pastlk{K'|_{\sigma}}{\vecv}$.
    Hence, we obtain that $J \setminus I\subseteq
    \pastlk{K'|_{\sigma}}{\vecv}$,
    and so \stmtref{pastlink-Kprime} holds.
\end{proof}

Using~\lemref{pastlink}, we see
why~$\tau$ cannot be the vertex $\min(\sigma)$ when performing
an LPDC. If $\tau = \min(\sigma)$, then
\begin{align*}
    \pastlk{K'|_{\sigma}}{\vecv} &=\{\vecj\in \{0,1\}^n\setminus\setzero \mid
    \min(\sigma) \preceq \vecv-\vecj\}\\
    &\quad \setminus\{\vecj \in \{0,1\}^n\setminus\setzero \mid \vecv-\vecj \preceq \min(\sigma)\}\\
    &=\{\vecj\in\{0,1\}^n\setminus\setzero \mid \min(\sigma)\preceq \vecv-\vecj \text{ and } \vecv-\vecj \succ \min(\sigma)\}\\
    &=\{\vecj\in\{0,1\}^n\setminus\setzero \mid \min(\sigma) \prec \vecv-\vecj\}\\
    &=\{\vecj\in\{0,1\}^n\setminus\setzero \mid \vecj \prec
    \vecv-\min(\sigma)\}.
\end{align*}
If $\vecv$ is the maximum vertex of $\sigma$, then we
obtain
$$ \pastlk{K'|_{\sigma}}{\vecv} =\{0,1\}^n\setminus \{ \binzero, \vecv-\min(\sigma) \}.$$
This computation gives us the following corollary, which we illustrate in
\figref{box-minvert} when $K$ is a single closed three-cube.

\begin{corollary}[Caution for a $(\min(\sigma),\sigma)$-Collapse]\label{cor:minvertexcollapse}
    Let $(K,\cubes{K})$ be a directed Euclidean cubical complex in~$\R^n$.
    Let $\sigma \in \cubes{K}$, $\tau = \min(\sigma)$, and $\vecv \in \verts{\sigma}$.
    If $\tau$ is a free face and $K'$ is the~$(\tau,\sigma)$-collapse, then the
    past link of $\vecv$ in $K'|_{\sigma}$ is:
    $$\{\vecj\in\{0,1\}^n\setminus\setzero \mid \vecj \prec \vecv-\min(\sigma)\}.$$

    In particular, if $\vecv = \max(\sigma)$ and $k=\dim(\sigma)$, then
    the past link  is the complete complex on $k$ elements before the collapse, and, after the
    collapse, it is homeomorphic to $\S^{k-2}$. Thus,~$(\tau,\sigma)$ is not an LPDC pair.

\end{corollary}
\begin{figure}[th]
    \centering
    \subfloat[Initial Complex\label{subfig:box-minvert-with}]{%
        \includegraphics[height=1in]{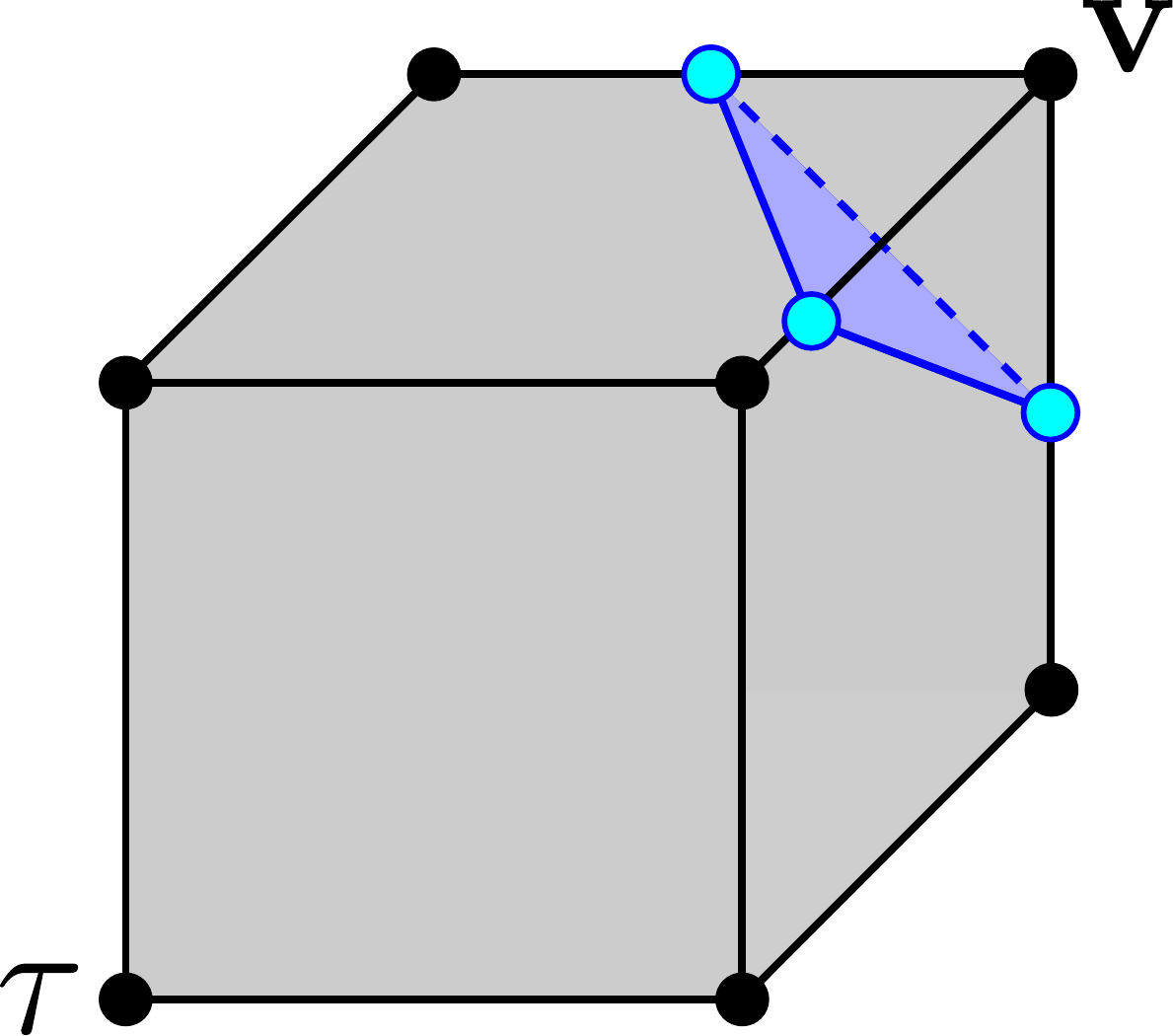}%
    }\hfil
    \subfloat[After Collapse\label{subfig:box-minvert-without}]{%
        \includegraphics[height=1in]{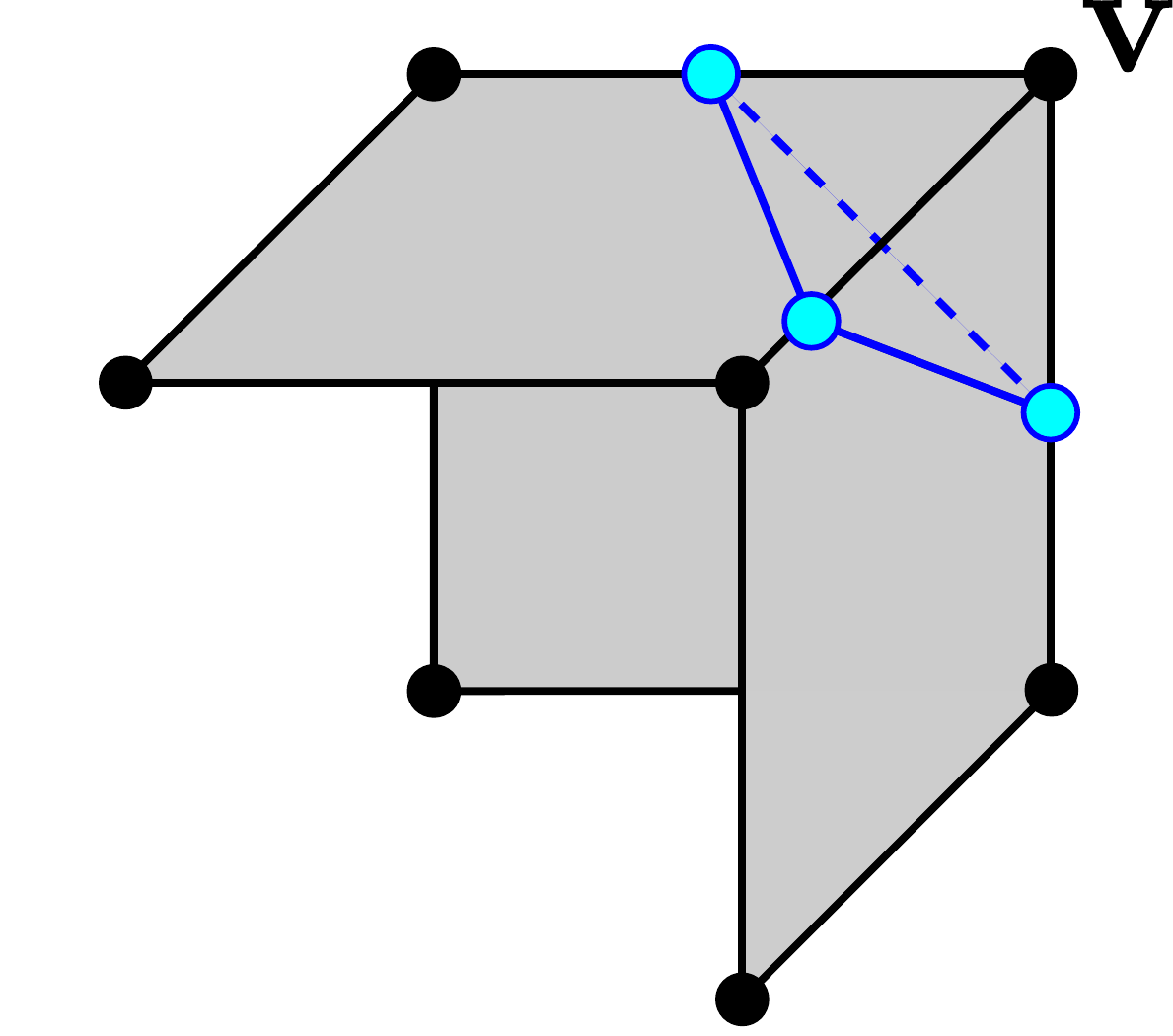}%
    }
    \caption{Removing the minimum vertex of a cube. Consider the directed
    Euclidean cubical complex
    in \protect\subref{subfig:box-minvert-with},
    which as a subset of~$\R^3$ is a single closed three-cube; call this three-cube
    $\sigma$.
    Letting $\tau = \min(\sigma)$, we observe that the past link of
    $\vecv=\max(\sigma)$ is
    contractible before the $(\tau,\sigma)$-collapse and is homeomorphic to
    $\S^1$ after the collapse. Thus, the past links before and after the
    collapse are not homotopy equivalent, and so
    this collapse is not an~LPDC.
    }
    \label{fig:box-minvert}
\end{figure}

The following lemma shows under which condition a directed Euclidean cubical collapse
induces a simplicial collapse in the past link.

\begin{lemma}[Vertex Collapses that Induce Simplicial Collapse of  Past Links]\label{lem:simp-collapse}
    Let $(K,\cubes{K})$ be a directed Euclidean cubical complex in $\R^n$.
    Let $\sigma \in \cubes{K}$ and~$\tau,\vecv \in \verts{\sigma}$
    such that $\tau \preceq \vecv$ and $\tau \neq \min(\sigma)$.
    If $\tau$ is a free face of $\sigma$ and $K'$ is the $(\tau,\sigma)$-collapse,
    then $\pastlk{K'}{\vecv}$ is the $(\vecv-\tau)$-collapse
    of $\pastlk{K}{\vecv}$ in the simplicial setting.
\end{lemma}

\begin{proof}
    Consider $\downset{K}{\vecv}$.
    Since $\tau, \vecv \in \verts{\sigma}$ and $\sigma$ is maximal in $K$, we
    know~$[\min(\sigma),\vecv]$ and~$[\tau, \vecv]$ are elementary cubes
    in~$\downset{K}{\vecv}$.
    Since $\tau$ is a free face of $\sigma$, we further
    know that $[\min(\sigma),\vecv]$ is the only maximal proper coface of $[\tau, \vecv]$ in
    $\downset{K}{\vecv}$.
    By definition of past link (\defref{pl}), we then
    have that~$\vecv-\min(\sigma)$ and $\vecv-\tau$ are simplices
    in~$\pastlk{\downset{K}{\vecv}}{\vecv}$,
    and~$\vecv-\min(\sigma)$ is the only maximal proper
    coface of~$\vecv-\tau$ in $\pastlk{\downset{K}{\vecv}}{\vecv}$.
    Hence, $\vecv-\tau$ is free in~$\pastlk{\downset{K}{\vecv}}{\vecv}$.
    Moreover, $\pastlk{\downset{K'}{\vecv}}{\vecv}$ is the $(\vecv-\tau)$-collapse of
    $\pastlk{\downset{K}{\vecv}}{\vecv}$. One can see this by using
    \lemstmtref{pastlink}{Kprime} by which $\pastlk{\downset{K'}{\vecv}}{\vecv}$ can be characterized as the $(\vecv-\tau)$-collapse of
    $\pastlk{\downset{K}{\vecv}}{\vecv}$.

    By \lemstmtref{plprops}{upset}, we know
    that~$\pastlk{K}{\vecv}=\pastlk{\downset{K}{\vecv}}{\vecv}$
    and that~$\pastlk{K'}{\vecv}=\pastlk{\downset{K'}{\vecv}}{\vecv}$, which concludes this
    proof.
\end{proof}

Next, we prove two lemmas concerning relationships of the past link of a vertex in the original
directed Euclidean cubical complex and in the collapsed directed Euclidean cubical complex.
These relationships depend on where $\vecv$ is located with respect to $\tau$.
In the first lemma, we consider the case where~$\min(\tau) \not\preceq \vecv$,
and we present a  sufficient condition for past links in $K$
and the $(\tau,\sigma)$-collapse to be equal.
See~\figref{box-v-uncomparable} for an example that illustrates the result of
this lemma.
\begin{figure}[t!b!]
    \centering
    \subfloat[Initial Complex\label{subfig:box-v-uncomparable-before}]{%
        \includegraphics[height=1in]{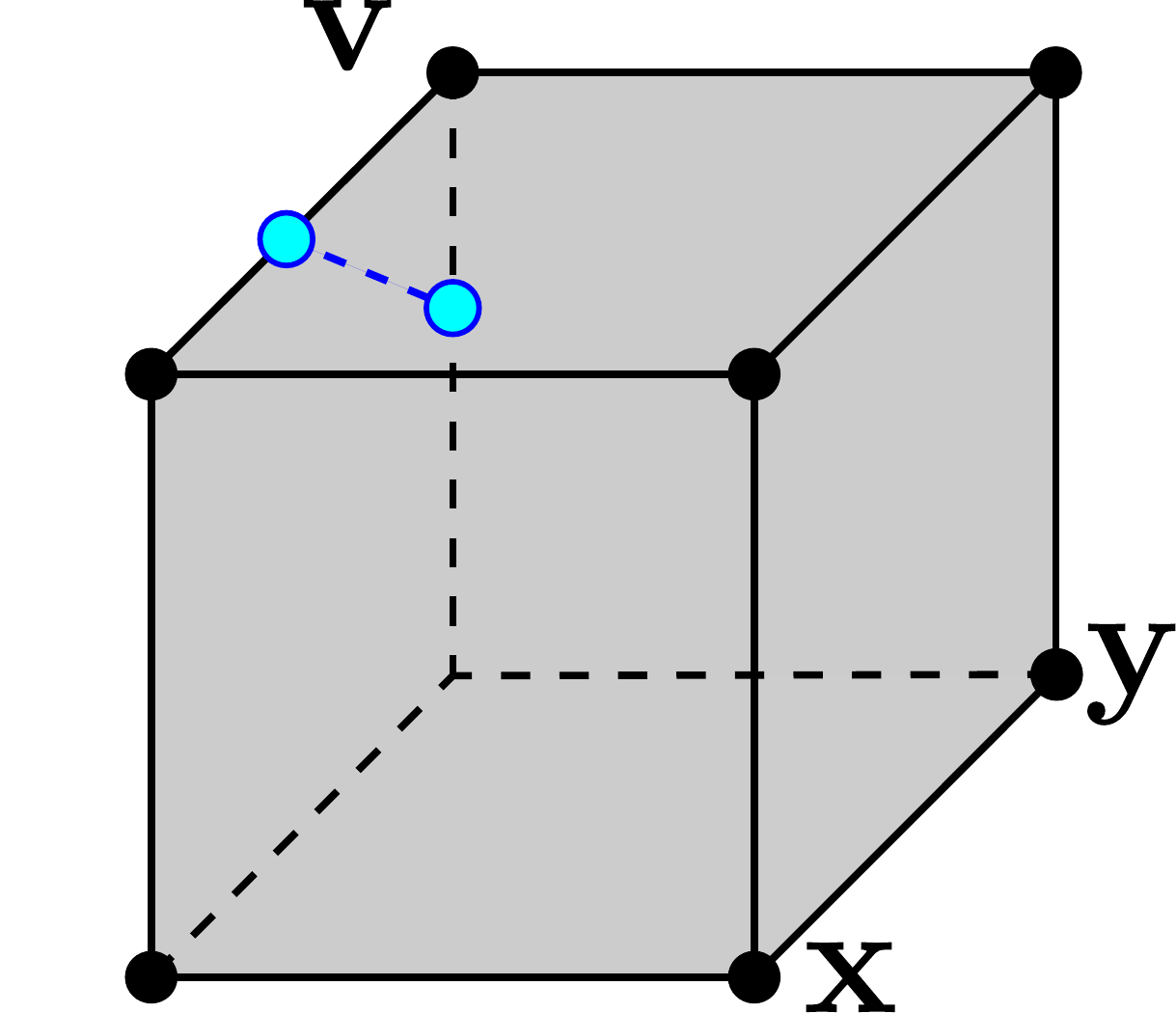}}%
    \hfil
    \subfloat[After Collapse\label{subfig:box-v-uncomparable-after}]{%
        \includegraphics[height=1in]{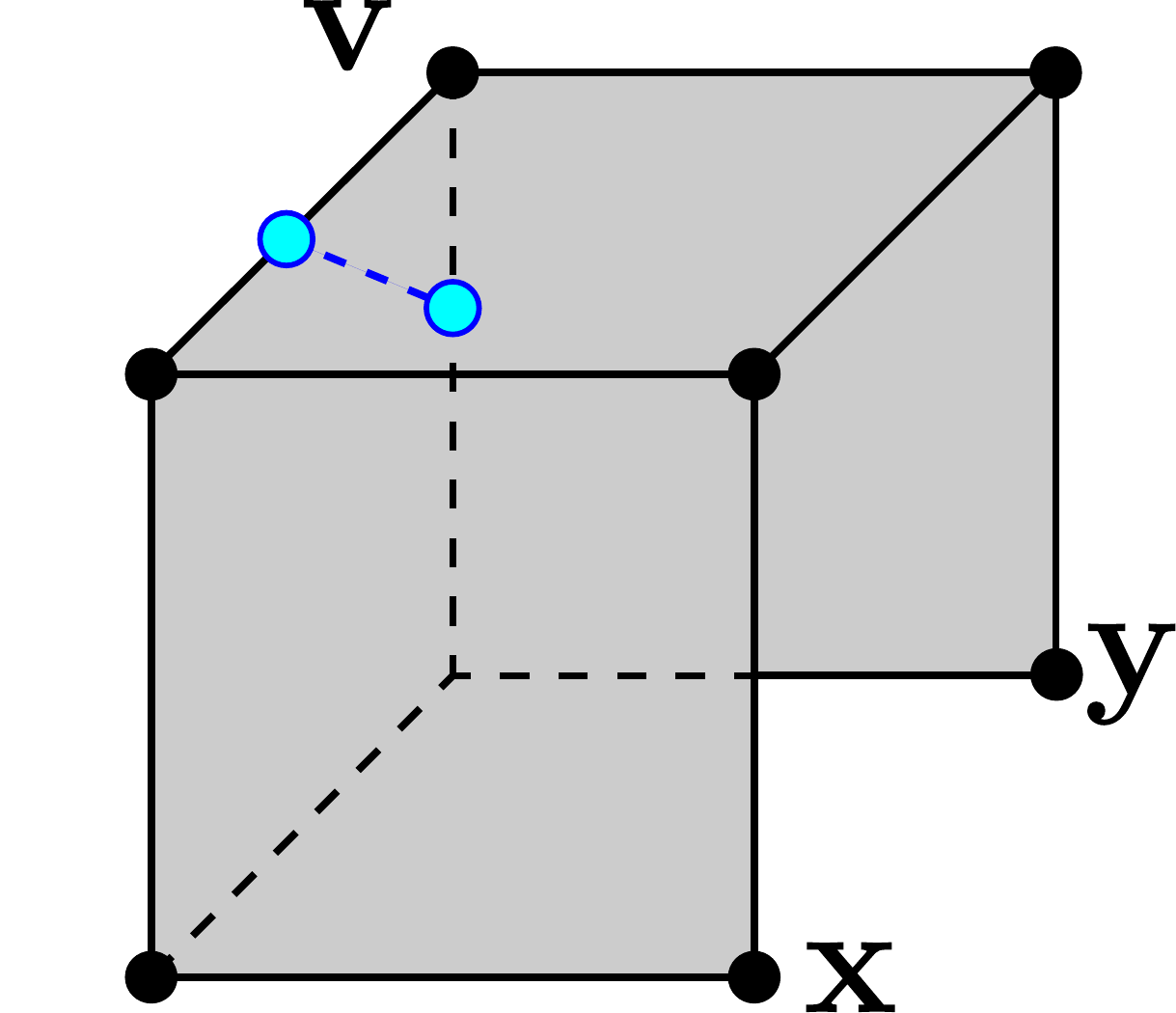}}%
    \caption{
        Past link of an ``uncomparable'' vertex before and after a collapse.
        Consider the directed Euclidean cubical complex shown,
        comprising a single three-cube $\sigma$ and all of its faces.
        Let
        $\tau=[\vecx,\vecy]$.
        Since $\vecv$ and $\max(\tau)=\vecy$ are not comparable, by
        \lemref{v_less_than_tau}, the past link of $\vecv$ is the same before
        and after the collapse.  Indeed, we see that this is the case for this
        example.  The past link of $\vecv$ is the complete complex on two
        vertices, both before and after.
    }
    \label{fig:box-v-uncomparable}
\end{figure}

\begin{lemma}[Condition for Past Links in ${K}$ and ${K'}$ to be Equal]\label{lem:v_less_than_tau}
    Let $(K,\cubes{K})$ be a directed Euclidean cubical complex in $\R^n$.
    Let $\tau, \sigma \in \cubes{K}$ such that $\tau$ is a face of $\sigma$.
    If $\tau$ is a free face of $\sigma$
    and $K'$ is the~$(\tau,\sigma)$-collapse,
    then, for all~$\vecv \in \verts{K}$ such that $\max(\tau) \not\preceq \vecv$,
    we have \mbox{$\pastlk{K}{\vecv}=\pastlk{K'}{\vecv}$.}
\end{lemma}
\begin{proof}
    By \lemstmtref{plprops}{subcomplex}, we have $\pastlk{K'}{\vecv} \subseteq
    \pastlk{K}{\vecv}$.  Thus, we only need to show
    $\pastlk{K}{\vecv} \subseteq \pastlk{K'}{\vecv}$.
    Suppose $\vecj \in \pastlk{K}{\vecv}$. By the definition of the past link
    (see \defref{pl}), we know that
    $[\vecv-\vecj,\vecv]$ is an elementary cube in $K$.
    By assumption, ~$\max(\tau) \not\preceq \vecv$.
    Thus, by
    \eqnref{directedCollapse}, $[\vecv-\vecj,\vecv]$ is not removed from $K$ and
    thus is an elementary cube in $K'$.
    Thus, $\vecj \in  \pastlk{K'}{\vecv}$.
\end{proof}

In the following lemma, we consider the case where $\max(\tau) \preceq \vecv$, and we
present a  sufficient condition for past links in the $(\tau,\sigma)$-collapse
and the~$(\min(\tau),\sigma)$-collapse  to be equal. See~\figref{box-collapse} for an example that illustrates this result.

\begin{figure}[th]
    \centering
    \subfloat[Original Complex\label{subfig:box-collapse-before}]{%
        \includegraphics[height=1in]{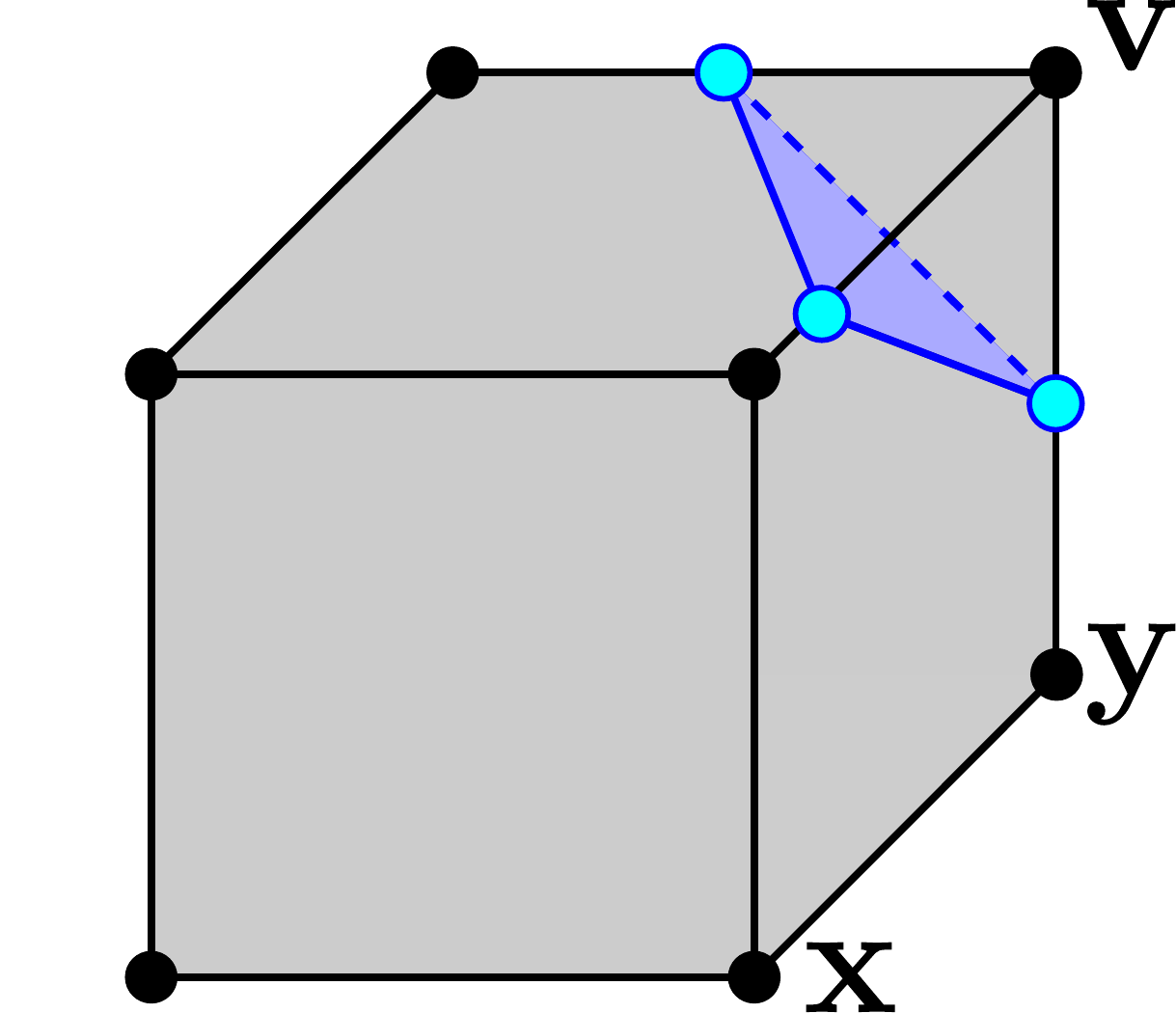}}%
        \hfil
    \subfloat[The $(\tau,\sigma)$-Collapse\label{subfig:box-collapse-edge}]{%
        \includegraphics[height=1in]{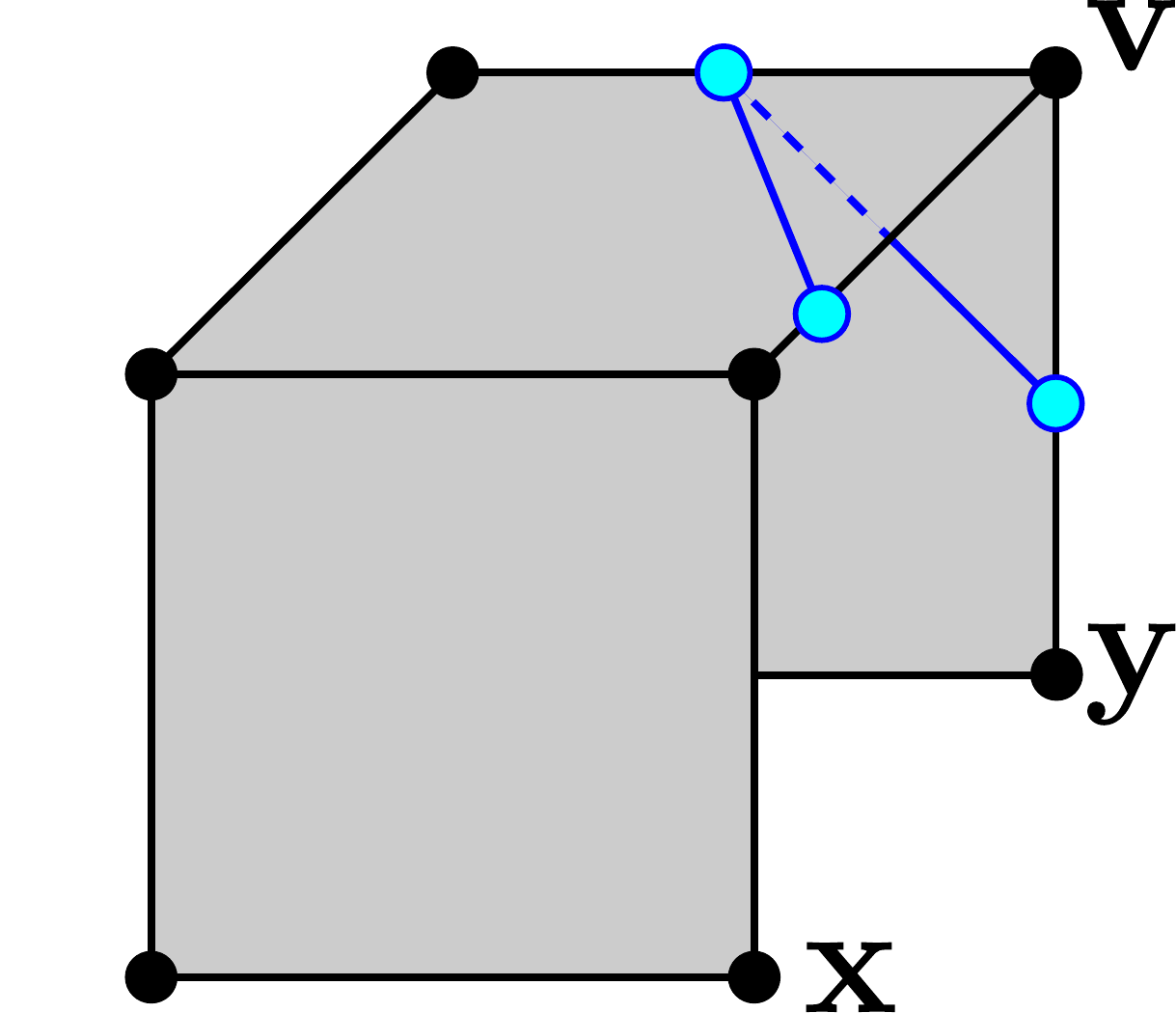}}%
        \hfil
    \subfloat[The $(\vecx,\sigma)$-Collapse\label{subfig:box-collapse-vertex}]{%
        \includegraphics[height=1in]{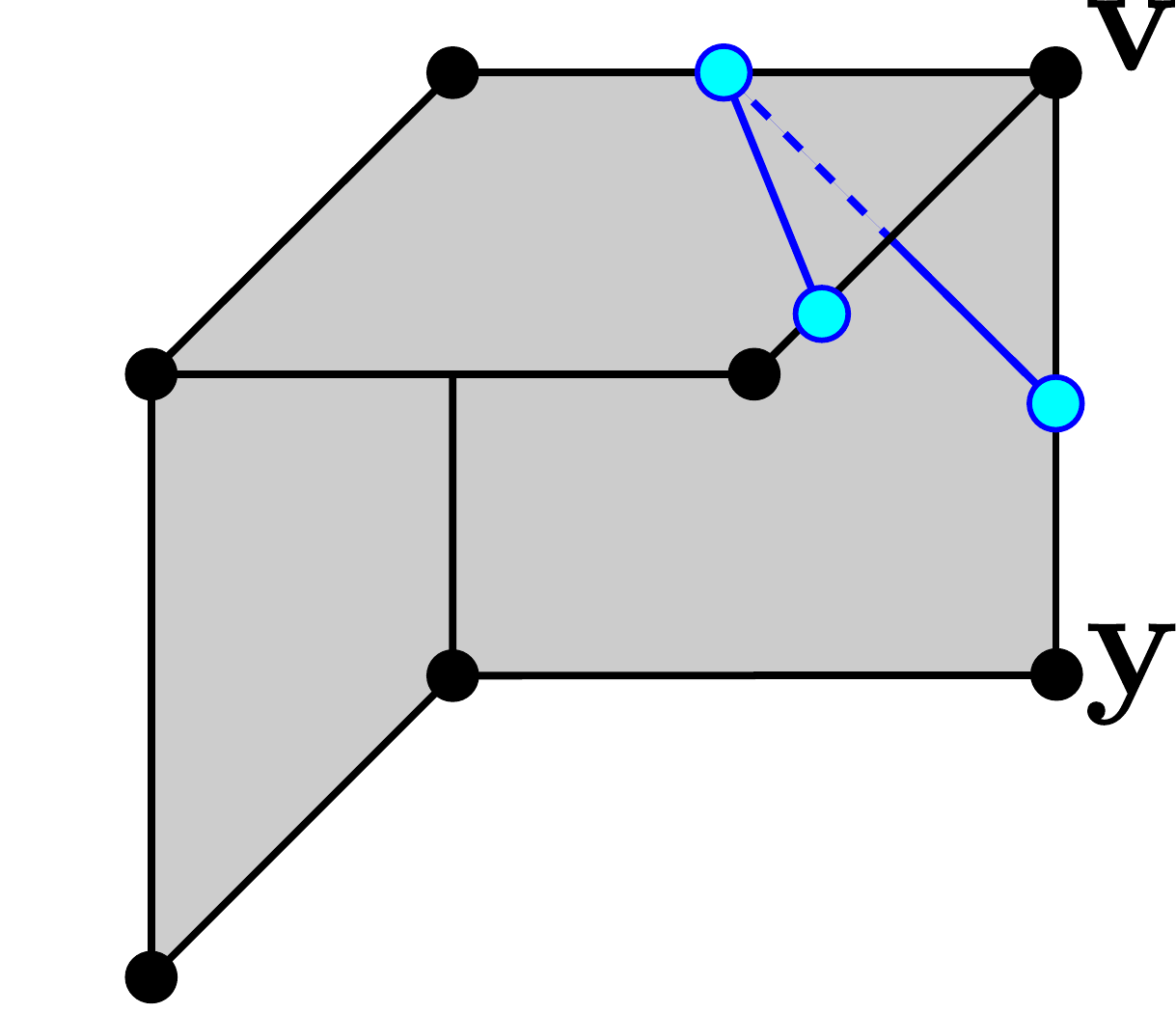}}%
    \caption{
        Two collapses with same past links.
        For example, in the directed Euclidean cubical complex $K$ shown in
        \protect\subref{subfig:box-collapse-before}, let $\sigma$ be the three-cube, and
        let~$\tau=[\vecx,\vecy]$.  We look at the past link of the vertex~$\vecv$.
        In the original directed Euclidean cubical complex, the past link of $\vecv$ is
        the complete complex on three vertices.
        By \lemref{past-links-equal-khat}, the past link of $\vecv$ is the same
        in both the $(\tau,\sigma)$-collapse and the $(\vecx,\sigma)$-collapse since $\max(\tau)=\vecy
        \preceq \vecv$. By \lemref{vertex-collapse}, we also know that the past
        links of $\vecv$ in $K$ and the $(\vecx,\sigma)$-collapse are homotopy equivalent.
        Indeed, we see that this is the case.
    }
    \label{fig:box-collapse}
\end{figure}

\begin{lemma}[Comparing Past Links in a General Collapse with Past Links in a
    Vertex Collapse]\label{lem:past-links-equal-khat}
    Let $(K, \cubes{K})$ be a directed Euclidean cubical complex in $\R^n$ such
    that there exists cubes $\tau, \sigma \in \cubes{K}$ with $\min(\tau)$ a free face of
    $\sigma$. Let $K'$ be the $(\tau, \sigma)$-collapse and let $\Khat$ be the
    $(\min(\tau), \sigma)$-collapse. If $\vecv \in \verts{K'}$ and $\max(\tau) \preceq \vecv$,
    then~$\vecv \in \verts{\Khat}$ and~$\pastlk{K'}{\vecv}=\pastlk{\Khat }{\vecv}$.
\end{lemma}

\begin{proof}
    We first show $\vecv \in \verts{\Khat}$. If $\tau$ is a zero-cube (and hence
    in~$\verts{K}$), then $K'=\Khat$, which means that $\vecv \in \verts{\Khat}$.
    On the other hand, if~$\tau$ is not a zero-cube, then
    we have $\min(\tau)\prec
    \max(\tau)\preceq \vecv$.  In particular,~$\min(\tau) \neq \vecv$.
    And so, by definition of $\Khat$ as a $(\min(\tau), \sigma)$-collapse and
    since $\vecv \in \cubes{K}$, we conclude that $\vecv \in \Khat$.

    Next, we show $\pastlk{K'}{\vecv}=\pastlk{\Khat }{\vecv}$.
    By \lemstmtref{plprops}{subcomplex}, we have~$\pastlk{\Khat}{\vecv} \subseteq
    \pastlk{K'}{\vecv}$.  Thus, what remains to be proven is
    \mbox{$\pastlk{K'}{\vecv} \subseteq \pastlk{\Khat}{\vecv}$.}
    Let $\vecj\in  \pastlk{K'}{\vecv}$. By definition of the past link
    (\defref{pl}), we
    know that~$[\vecv-\vecj,\vecv] \subseteq K'$.
    Consider two cases: $\vecv -\vecj \preceq \min(\tau)$ and $\vecv -\vecj
    \not\preceq \min(\tau)$.

    Case 1 ($\vecv -\vecj \preceq \min(\tau)$):
    Since $\vecv-\vecj \preceq \min(\tau)
    \preceq \max(\tau) \preceq \vecv$, we know
    that $\overline{\tau} \subseteq [\vecv-\vecj,\vecv]$.  Thus,
    by \eqnref{directedCollapse}, we have $[\vecv-\vecj,\vecv] \nsubseteq K'$, which
    is a contradiction.  So, Case 1 cannot happen.

    Case 2 ($\vecv -\vecj \not\preceq \min(\tau)$): If $\vecv -\vecj \not\preceq
    \min(\tau)$, then, by the definition of a~$(\min(\tau),\sigma)$-collapse
    in \defref{directedCollapse},
    we know that~$[\vecv-\vecj,\vecv] \subseteq \Khat$
    and thus~$\vecj \in \pastlk{\Khat}{\vecv}$.

    Hence, $\pastlk{K'}{\vecv} \subseteq \pastlk{\Khat}{\vecv}$. Since we have
    both subset inclusions, we conclude~$\pastlk{K'}{\vecv}=\pastlk{\Khat}{\vecv}$.
\end{proof}

In general, the minimal vertex of $\tau$ is not free in $K$ and hence, there is no vertex collapse.
In the main theorem, the previous lemma is applied to a subcomplex of $K$;
specifically, it is applied to the restriction to the unit cube corresponding to $\sigma$,
where all vertices, inculding $\min(\tau)$ are then free. The results  carry over to $K$.

The next result states that vertex collapses result in homotopy equivalent past links as long as we
are not collapsing the minimum vertex of the directed Euclidean cubical complex.

\begin{lemma}[Past Links in a Vertex Collapse]\label{lem:vertex-collapse}
    Let $(K,\cubes{K})$ be a directed Euclidean cubical complex in $\R^n$.
    Let $\sigma \in \cubes{K}$ and let~$\tau \in \verts{\sigma}$ such that $\tau
    \neq \min(\sigma)$.
    Let~$\vecv \in \verts{K}$ with $\vecv \neq \tau$.
    If $\tau$ is a free face of $\sigma$ and $K'$ is the $(\tau,\sigma)$-collapse,
    then~$\pastlk{K}{\vecv} \simeq \pastlk{K'}{\vecv}$.
\end{lemma}

\begin{proof}
    We consider three cases:

    Case 1 ($\vecv \notin \verts{\sigma}$):
    By definition of past link (\defref{pl}), if $\vecv \notin \verts{\sigma}$,
    then the past links $\pastlk{K}{\vecv}$ and $\pastlk{K'}{\vecv}$ are
    equal.

    Case 2 ($\tau \not\preceq \vecv$):
    By \lemref{v_less_than_tau}, if $\tau=\max(\tau) \not\preceq \vecv$, again we have
    equality of the past links $\pastlk{K}{\vecv}$
    and~$\pastlk{K'}{\vecv}$.

    Case 3 ($\vecv \in \verts{\sigma}$ and $\tau \preceq \vecv$):
    By \lemref{simp-collapse}, we know that
    $\pastlk{K'}{\vecv}$ is the $\vecv-\tau$-collapse of
    $\pastlk{K}{\vecv}$ in the simplicial setting. Since simplicial
    collapses preserve the homotopy type
    (see e.g.,~\cite[Proposition 6.14]{kozlovcombinatorial2007}), we
    conclude~$\pastlk{K}{\vecv}\simeq \pastlk{K'}{\vecv}$.

\end{proof}

We give an example of \lemref{vertex-collapse} in \figref{box-collapse} by showing how
the LPDC induces a simplicial collapse on past links.

Lastly, we are ready to prove the main result.
\begin{theorem}[Main Theorem]\label{thm:collapsingpairs}
    Let $(K,\cubes{K})$ be a directed Euclidean cubical complex in $\R^n$ such
    that there exist cubes $\tau, \sigma \in \cubes{K}$ with $\tau$ a free face of
    $\sigma$.
    Then, $(\tau,\sigma)$ is an LPDC pair
    if and only if~$\min(\sigma) \notin \verts{\tau}$.
\end{theorem}
\begin{proof}
    Let~$\vecv=\max(\sigma)$ and~$k=\dim(\sigma)$.
    Let $(K', \cubes{K}')$ be the $(\tau, \sigma)$-collapse of $K$.
    Let $(L,\cubes{L})$ be the cubical complex such that $L = K|_{\sigma}$.  Since $\sigma
    \in K$, we know $L=\overline{\sigma}$ (i.e., $L$ is a unit cube).
    Since $L$ is a single unit cube and $\sigma$ is a maximal elementary cube,
    all proper faces of $\sigma$, including
    $\tau$ and $\min(\tau)$, are free faces in $L$.
    Thus, let $(L', \cubes{L}')$
    be the~$(\tau, \sigma)$-collapse of $L$,
    and let $(\widehat{L},\widehat{\cubes{L}})$ be the~$(\min(\tau),
    \sigma)$-collapse of~$L$.

    We first prove the forward direction by contrapositive (if \mbox{$\min(\sigma)
    \in \verts{\tau}$}, then $(\tau,\sigma)$ is not an LPDC pair).
    Assume~$\min(\sigma) \in \verts{\tau}$.
    By \corref{minvertexcollapse}, we obtain~$\pastlk{L}{\vecv}$ is
    homeomorphic to~$\B^{k-1}$ and $\pastlk{\widehat{L}}{\vecv}$ is homeomorphic
    to~$\S^{k-2}$.
    Since $\min(\sigma) \in \verts{\tau}$, we know that \mbox{$\min(\sigma)=\min(\tau)$.}
    Since~$\tau$ is a face of~$\sigma$, we know~$\max(\tau) \preceq
    \max(\sigma)=\vecv$.
    Since $\min(\sigma)=\min(\tau) \in \verts{\tau}$ and since~$\tau$
    is a proper face of $\sigma$, we know that
    \mbox{$\vecv \neq \max(\tau)$.}
    Thus,~$\vecv \in \verts{L'}$.
    Applying \lemref{past-links-equal-khat}, we
    obtain $\pastlk{L'}{\vecv}=\pastlk{\widehat{L}}{\vecv}$.
    Putting this all together, we~have:
    $$\pastlk{L}{\vecv} \simeq \B^{k-1} \not\simeq \S^{k-2}
    \simeq \pastlk{\widehat{L}}{\vecv} = \pastlk{L'}{\vecv},$$
    and so $\pastlk{L}{\vecv} \not\simeq \pastlk{L'}{\vecv}$.

    Since no faces of $\sigma$ are in $\cubes{K} \setminus \cubes{L}$,
    the past link of~$\vecv$ remains the same outside of $L$ in both~$K$
    and~$K'$.
    Thus,~$\pastlk{K}{\vecv} \not\simeq \pastlk{K'}{\vecv}$ and so we conclude
    that~$(\tau, \sigma)$ is not an LPDC pair, as was to be~shown.

    Next, we show the backwards direction. Suppose $\min(\sigma) \not\in
    \verts{\tau}$. Let~$\vecv \in
    \verts{K'}$, and consider two cases:
    $\max(\tau) \npreceq \vecv$ and $\max(\tau) \preceq \vecv$.

    Case 1 ($\max(\tau) \npreceq \vecv$):
    By~\lemref{v_less_than_tau}, we have $\pastlk{K}{\vecv}=\pastlk{K'}{\vecv}$.
    Hence,~$\pastlk{K}{\vecv}\simeq \pastlk{K'}{\vecv}$. Since~$\vecv$ was
    arbitrarily chosen, we conclude that~$(\tau,\sigma)$ is an LPDC pair.

    Case 2 ($\max(\tau)\preceq \vecv$):
    By~\lemref{past-links-equal-khat}, we have that
    $\pastlk{L'}{\vecv}= \pastlk{\widehat{L}}{\vecv}$.
    Since~$\min(\sigma) \not\in \verts{\tau}$, we know that $\min(\tau)\neq \min(\sigma)$.
    Applying \lemref{vertex-collapse}, we obtain
    $\pastlk{L}{\vecv} \simeq \pastlk{\widehat{L}}{\vecv}$.
    Again, since no faces of $\sigma$ are removed from~$\cubes{K}$ and~$\cubes{K'}$
    to obtain $\cubes{L}$ and $\cubes{L'}$,
    the past link of~$\vecv$ remains the same outside of~$L$ in both~$K$
    and~$K'$.
    Thus,~$\pastlk{K}{\vecv} \simeq  \pastlk{K'}{\vecv}$.
    Since~$\vecv$ was
    arbitrarily chosen, we conclude that~$(\tau,\sigma)$ is an LPDC pair.
\end{proof}

\section{Preservation of Spaces of Dipaths}\label{sec:pathspaces}

In \cite{belton2020towards}, we proved several results on the relationships between past links and spaces of
dipaths. One result, \thmref{partial-contractability}, states that for a directed Euclidean cubical
complex with a minimum vertex, if all past links are contractible, then all spaces of
dipaths starting at that minimum vertex are also contractible. If we
start with a directed Euclidean cubical complex with a minimum vertex that has all contractible past links,
then all spaces of dipaths from the minimum vertex are contractible by this theorem. We explain how those relationships extend to the LPDC setting in this section.

Applying an LPDC preserves the homotopy type of past links by definition. Hence,
applying the theorem again, we see that any LPDC also has contractible dipath spaces
from the minimum vertex. Notice that the minimum vertex is not removed in an LPDC, since it is a vertex and
minimal in all cubes  containing it (including the maximal cube).
We give an example of this in \exref{contractible-ps}.

\begin{example}[$3\times 3$ filled grid]\label{ex:contractible-ps}
    Let $K$ be the $3\times 3$ filled grid. For all $\vecv \in \verts{K}$,
    $\pastlk{K}{\vecv}$ is contractible. By~\thmref{partial-contractability}, this
    implies that all spaces of dipaths starting at $\zero$ are contractible. Applying an LPDC such as the edge
$[(1,3),(2,3)]$ results in contractible past links in $K'$ and so all spaces of dipaths in $K'$
are also contractible.  See \figref{contract-collapse}. We can generalize this example to
    any $k^d$ filled grid where $k,d \in \mathbb{N}$.

\begin{figure}[th]
    \centering
    \subfloat[Initial Complex\label{subfig:contractibility-collapse1}]{%
        \includegraphics[height=1in]{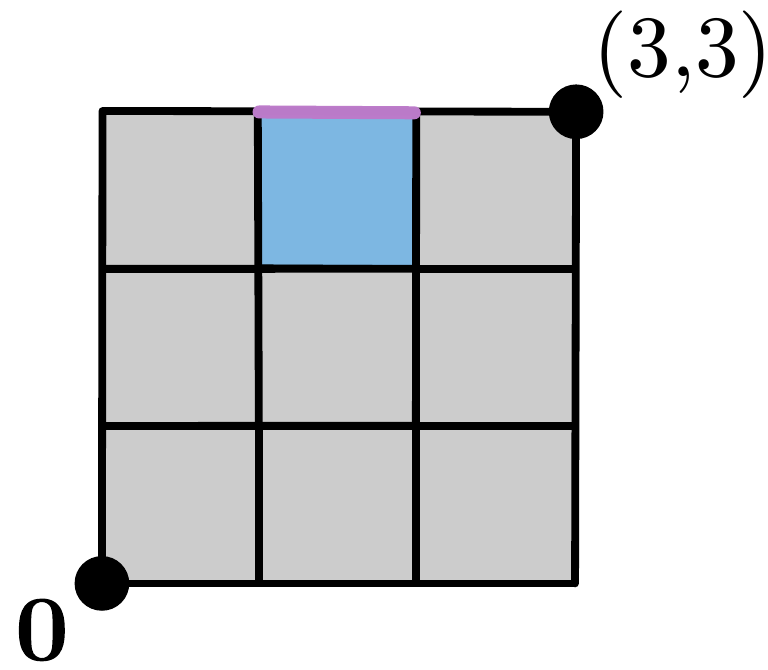}%
    }\hfil
    \subfloat[After Collapse\label{subfig:contractibility-collapse2}]{%
        \includegraphics[height=1in]{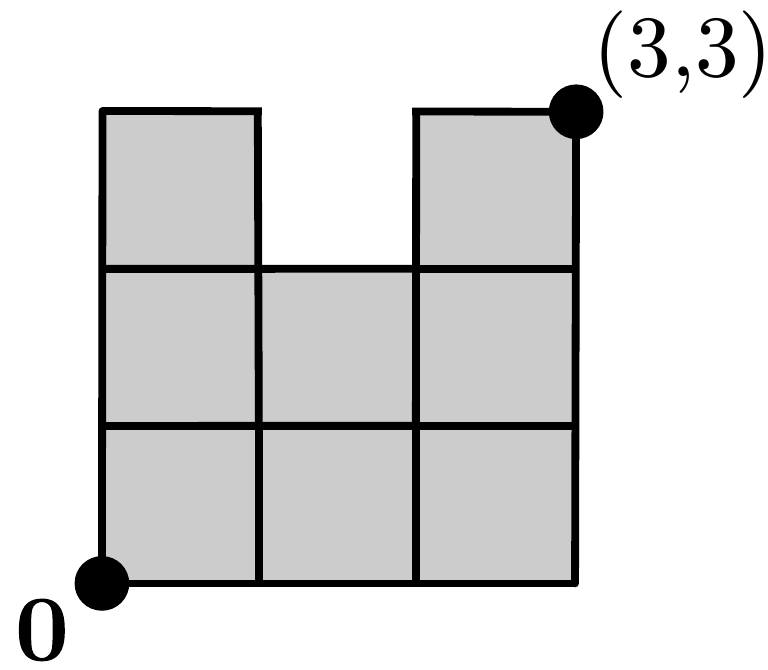}%
    }
    \caption{\protect\subref{subfig:contractibility-collapse1} The $3\times 3$ filled grid has contractible past links and
    dipath spaces. The pair comprising of the purple edge $[(1,3),(2,3)]$
    and the blue square~$[(1,2), (2,3)]$ is an LPDC pair. \protect\subref{subfig:contractibility-collapse2} The result of performing the LPDC.  All past links are contractible and so all dipath spaces are also contractible. }
\label{fig:contract-collapse}
\end{figure}

\end{example}

An analogous result holds for connectedness (\thmref{connected}). If we start with a directed
Euclidean cubical complex such that all past links are connected, then all dipath spaces are
connected. Any LPDC results in a directed Euclidean cubical complex that also has connected
dipath spaces. See~\exref{connected-ps}.

\begin{example}[Outer Cubes of the $5\times 5\times 5$ Grid]\label{ex:connected-ps}
    Let $K=[0,5]^3\setminus[1,4]^3$, which, as an undirected complex, is
    homeomorphic to a thickened two-sphere.
    For all $\vecv \in \verts{K}$, $\pastlk{K}{\vecv}$ is
    connected. By~\thmref{connected}, this implies that  for all $\vecv \in
    \verts{K}$, the space of
    dipaths~$\dipaths{K}{\zero}{\vecv}$ is connected. Applying an LPDC such as
    with the vertex~$(5,0,0)$ in the cube~$[(4,0,0), (5,1,1)]$ results in connected past links in $K'$ and so all spaces
    of dipaths $\dipaths{K'}{\zero}{\vecv}$ are connected. We can
    generalize this example to any $k^d$ grid where $d\geq 3$ and the inner
    cubes of dimension $d$ are~removed.
\end{example}

Both \thmref{partial-contractability} and \thmref{connected} have assumptions on the topology of past
links and results on the topology of spaces of dipaths from the minimum vertex. We may ask
if the converse statements are true. Does knowing the topology of spaces of dipaths from
the minimum vertex tell us anything about the topology of past links?
The converse to \thmref{partial-contractability} holds. To prove this, we first
need a
lemma whose proof appears in~\cite{ziemianski2016execution}.

\begin{lemma}[Homotopy Equivalence
    {\cite[Prop.~5.3]{ziemianski2016execution}}]\label{lem:homotopy-equivalence}
    Let $(K,\cubes{K})$ be a directed Euclidean cubical complex in $\R^n$.
    Let $\vecp,\vecq \in \Z^n$.
    If $\dipaths{K}{\vecp}{\vecq-\vecj}$ is contractible for all~$\vecj\in
    \pastlk{K}{\vecq}$, then $\dipaths{K}{\vecp}{\vecq} \simeq \pastlk{\upset{K}{\vecp}}{\vecq}$.
\end{lemma}

Thus, we obtain:

\begin{theorem}[Contractability]\label{thm:full-contractability}
    Let $(K,\cubes{K})$ be a directed Euclidean cubical complex in $\R^n$ that
    has a minimum vertex $\vecw$.
    The following two statements are equivalent:
    \begin{enumerate}
        \item For all~$\vecv \in \verts{K}$,
            the space of dipaths $\dipaths{K}{\vecw}{\vecv}$ is
            contractible.\label{stmt:full-contractability-dipaths}
        \item For all~$\vecv \in \verts{K}$,
            the past link $\pastlk{K}{\vecv}$ is contractible.\label{stmt:full-contractability-pastlinks}
    \end{enumerate}
\end{theorem}

\begin{proof}
   By \thmref{partial-contractability},
    we obtain \stmtref{full-contractability-pastlinks} implies
    \stmtref{full-contractability-dipaths}.

    Next, we show that
    \stmtref{full-contractability-dipaths} implies
    \stmtref{full-contractability-pastlinks}.
    Let $\vecv \in \verts{K}$.
    For all $\vecj \in \pastlk{K}{\vecv}$, the cube $[\vecv-\vecj, \vecv]$ is a
    subset of~$K$,
    which means that \mbox{$\vecv-\vecj \in \verts{K}$.}
    Thus, by assumption, all dipath spaces $\dipaths{K}{\vecw}{\vecv-\vecj}$ are contractible.
    By \lemref{homotopy-equivalence}, we know that
    \mbox{$\dipaths{K}{\vecw}{\vecv} \simeq \pastlk{\upset{K}{\vecw}}{\vecv}=\pastlk{K}{\vecv}$.}
    Again, since~$\vecv \in \verts{K}$, the dipath space~$\dipaths{K}{\vecw}{\vecv}$ is contractible.
    Therefore,~$\pastlk{K}{\vecv}$ is contractible.
\end{proof}

As a consequence of this theorem, we know that
if we start with a directed Euclidean cubical complex with
contractible dipath spaces starting at the minimum vertex, then any LPDC
also result in a directed Euclidean cubical complex with all contractible dipath spaces
starting at the minimum vertex, and vice versa.

\begin{corollary}[Preserving Directed Path Space Contractability]
    Let $(K,\cubes{K})$ be a directed Euclidean cubical complex in $\R^n$ that
    has a minimum vertex~$\vecw$.
    Let $\tau, \sigma \in \cubes{K}$ such that $\tau$ is a face of $\sigma$.
    If $\tau$ is a free face of $\sigma$, let
    $(K',\cubes{K}')$ be the~$(\tau,\sigma)$-collapse.
    If $K'$ is an LPDC of $K$, then
    the spaces of dipaths~$\dipaths{K}{\vecw}{\vecv}$
    are contractible for all $\vecv \in \verts{K}$ if and only
    if the spaces of dipaths~$\dipaths{K'}{\vecw}{\veck}$ are
    contractible for all $\veck \in \verts{K'}$.
\label{cor:contractible-collapse}
\end{corollary}
\begin{proof}
We start with the forwards direction by assuming that the spaces of dipaths
$\dipaths{K}{\vecw}{\vecv}$ are contractible for all $\vecv \in \verts{K}$.
\thmref{full-contractability} tells us that all past links $\pastlk{K}{\vecv}$ are contractible for
all $\vecv \in \verts{K}$. This implies that $\pastlk{K'}{\veck}$ is contractible for all
$\veck \in \verts{K'}$ because $K'$ is an LPDC of~$K$. Applying \thmref{full-contractability}
again, we see that all spaces of dipaths~$\dipaths{K'}{\vecw}{\veck}$ are contractible for all $\veck \in \verts{K'}$.

Next we prove the backwards direction by assuming that the spaces of dipaths
$\dipaths{K'}{\vecw}{\veck}$ are contractible for all $\veck \in \verts{K'}$. Let
$\vecv \in \verts{K}$. Either $\vecv \in \verts{K'}$ or $\vecv \notin \verts{K'}$.

Case 1 ($\vecv \in \verts{K'}$): By \thmref{full-contractability}, we know that
$\pastlk{K'}{\vecv}$ is contractible. Since $K'$ is an LPDC of $K$, then $\pastlk{K}{\vecv}$ is also contractible.

Case 2 ($\vecv \notin \verts{K'}$): If $\vecv \notin \verts{K}$, then $\tau$ is a vertex and
$\vecv = \tau$. Observe that $\pastlk{\overline{\sigma}}{\tau}$ is contractible since
$\overline{\sigma}$ is an elementary cube and~$\tau$ does not contain $\min(\sigma)$.
Furthermore, notice that $\pastlk{K}{\tau}=\pastlk{\sigma}{\tau}$ because $\tau$ is a
free face of $\sigma$. Hence, $\pastlk{K}{\tau}$ is contractible.

Therefore $\pastlk{K}{\vecv}$ is contractible for all $\vecv \in \verts{K}$.
Applying \thmref{full-contractability}, we get that $\dipaths{K}{\vecw}{\vecv}$ is contractible
for all $\vecv \in \verts{K}$.
\end{proof}

Using~\thmref{connected} and the partial converse to the connectedness
theorem~\cite[Theorem 3]{belton2020towards}, we get that any LPDC of a directed
Euclidean cubical complex with connected dipath spaces and reachable vertices
results in a directed Euclidean cubical complex with connected dipath spaces.

\begin{corollary}[Condition for LPDCs to Preserve Connectedness of All Directed Path Spaces]
\label{cor:connected-collapse}
    Let~$(K,\cubes{K})$ be a directed Euclidean cubical complex in $\R^n$ that
    has a minimum vertex $\vecw$.
    Let $(L,\cubes{L})=\reachcplx{K}{\vecw}$.
    Let $(\tau, \sigma)$ be an
    LPDC pair in~$L$, and let $L'$ be the $(\tau,\sigma)$-collapse.
    The spaces of dipaths in
    $\dipaths{L}{\vecw}{\veck}$ are connected for all $\vecv \in \verts{L}$ if and only if
    the spaces of dipaths~$\dipaths{L'}{\vecw}{\vecv}$ are connected for all
    $\vecv \in \verts{L'}$.
\end{corollary}

We note that reachability is a necessary condition. Below we give an example of
a directed Euclidean cubical complex $K$ that has all connected dipath spaces but
an LPDC yields a directed Euclidean cubical complex with a disconnected path space.

\begin{figure}[th]
    \centering
    \subfloat[Original Complex\label{subfig:bowlingball-K}]%
        {\includegraphics[width=.3\textwidth]{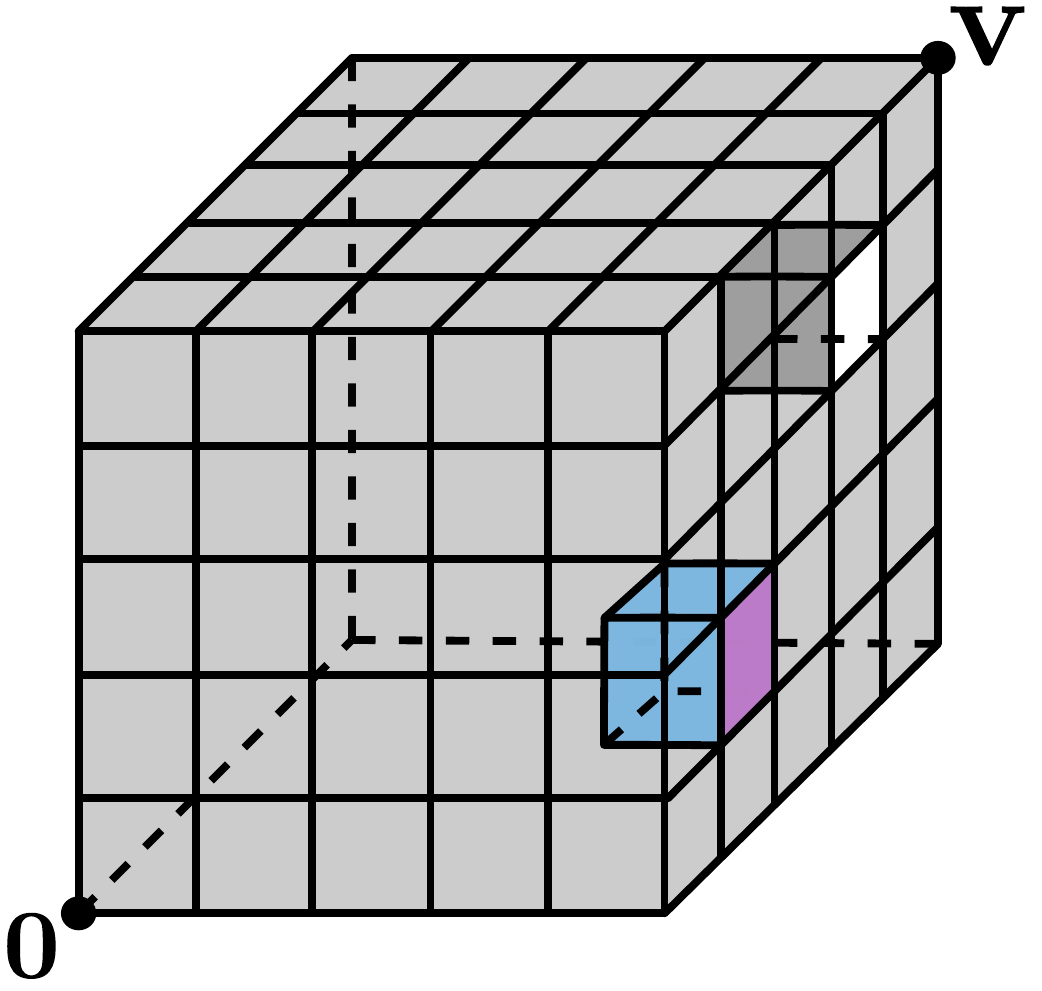}}
    \hfil
    \subfloat[After Collapse\label{subfig:bowlingball-Kprime}]%
        {\includegraphics[width=.3\textwidth]{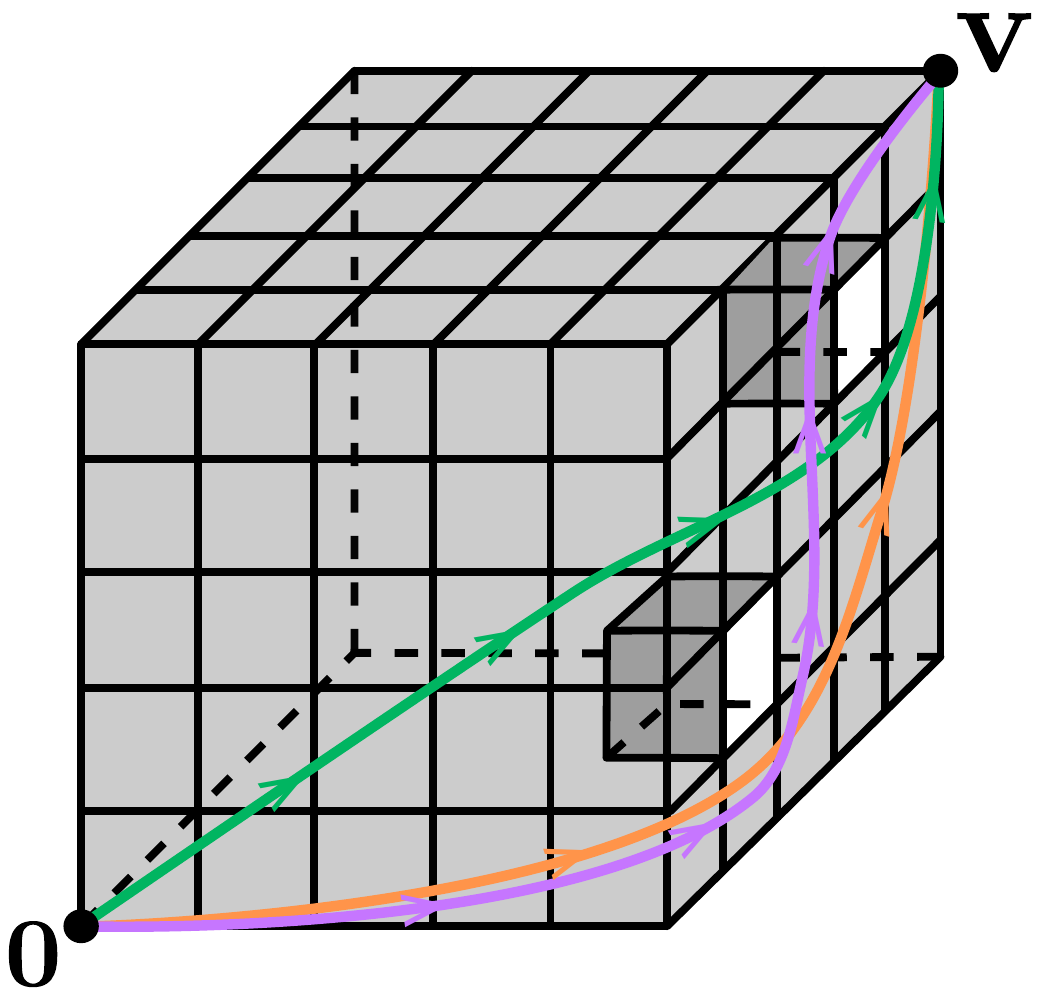}}

    \caption{
        The bowling ball before and after the collapse described in
        \exref{bowlingball}. Observe $\protect\dipaths{K}{\zero}{(5,5,5)}$
        has one connected component. Additionally,~$\sigma = [(4,1,1),(5,2,2)]$ (highlighted in blue)
        and \mbox{$\tau = [(5,1,1),(5,2,2)]$} (highlighted in purple) is an LPDC pair.
        After collapsing
        $(\tau, \sigma)$, $\protect\dipaths{K}{\zero}{(5,5,5)}$ changes from having one connected
        component to three connected components.
        The three connected components are represented by the three dipaths.}
        \label{fig:bowlingball}
\end{figure}
\begin{example}[Bowling Ball]\label{ex:bowlingball}
    Let $K$ be the complex resulting from taking boundary of the $5\times 5\times 5$ grid
    union the closed cubes~$[(4,1,1),(5,2,2)]$ and~$[(4,3,3),(5,4,4)]$, then removing the open cubes ~$[(4,3,3),(5,4,4)]$ and $[(5,3,3),(5,4,4)]$.
    See \subfigref{bowlingball}{K}.
    Notice that some vertices of $K$ are unreachable, for example, vertex $(4,1,1)$. Furthermore, all past links of vertices in $K$ are
    connected and so all dipath spaces starting at~$\zero$ are also connected. After performing an
    LPDC with $\tau = [(5,1,1),(5,2,2)]$ and~$\sigma = [(4,1,1),(5,2,2)]$, the dipath space
    between $\zero$ and $(5,5,5)$ changes from having one connected component to three connected
    components, as shown in the figure. This example shows that the
    reachability condition in~\corref{connected-collapse} is necessary for preserving connnectedness
    in LPDCs.
\end{example}

LPDCs can also preserve dihomotopy classes of dipaths starting at the minimum vertex of many directed
Euclidean cubical complexes that have disconnected past links. Recall the
Swiss flag as discussed in \figref{swissflag-collapse}. The Swiss flag has disconnected past links at $(3,4)$
and $(4,3)$, yet there exists a sequence of LPDCs that results in a directed Euclidean cubical
complex that highlights the two dihomotopy classes of dipaths between $\zero$
and $(5,5)$. \exref{window} gives another similar situation.

\begin{example}[Window]\label{ex:window}

\begin{figure}[th]
    \centering
    \subfloat[\label{subfig:window1}]{%
        \includegraphics[height=1in]{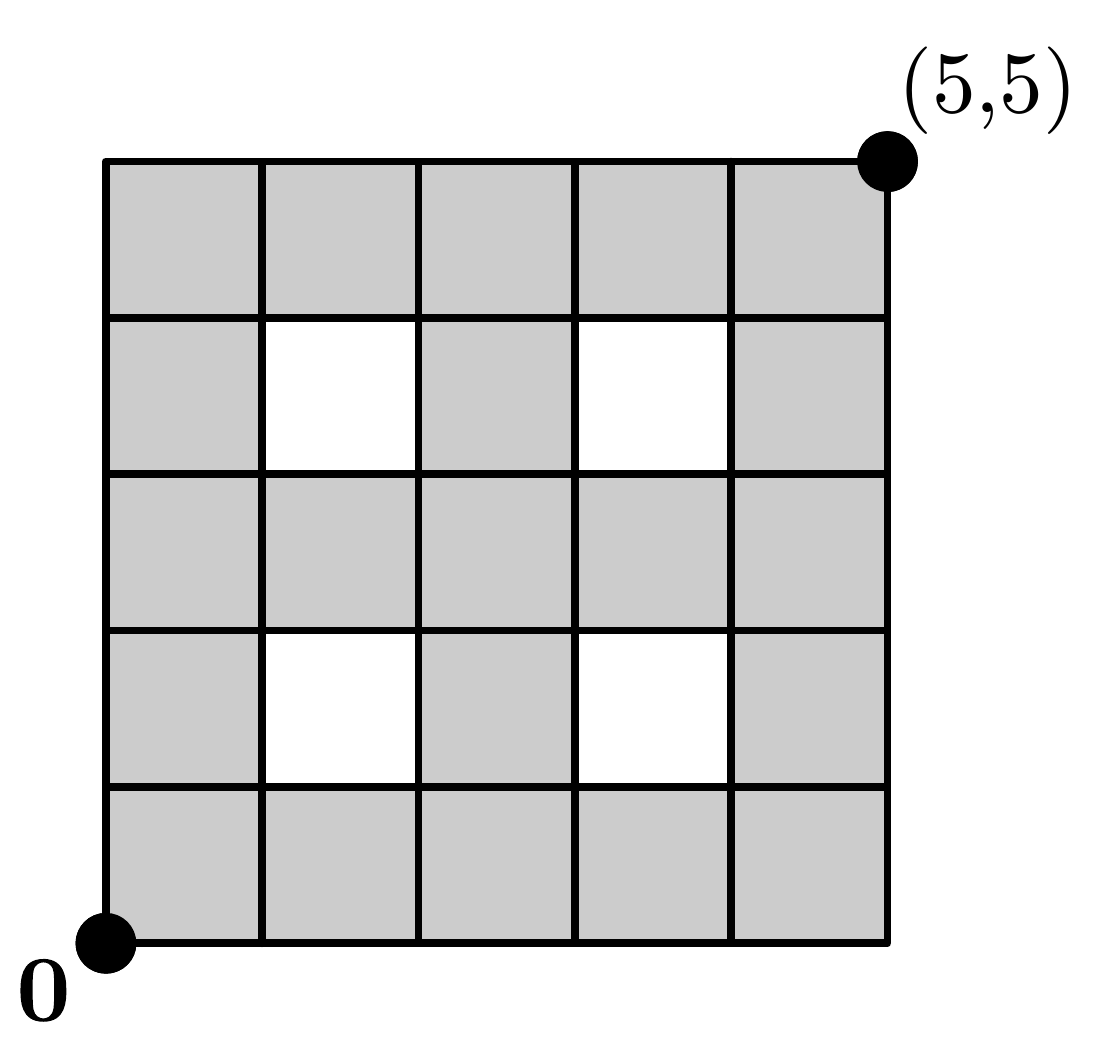}}%
    \subfloat[\label{subfig:window-collapse2}]{%
        \includegraphics[height=1in]{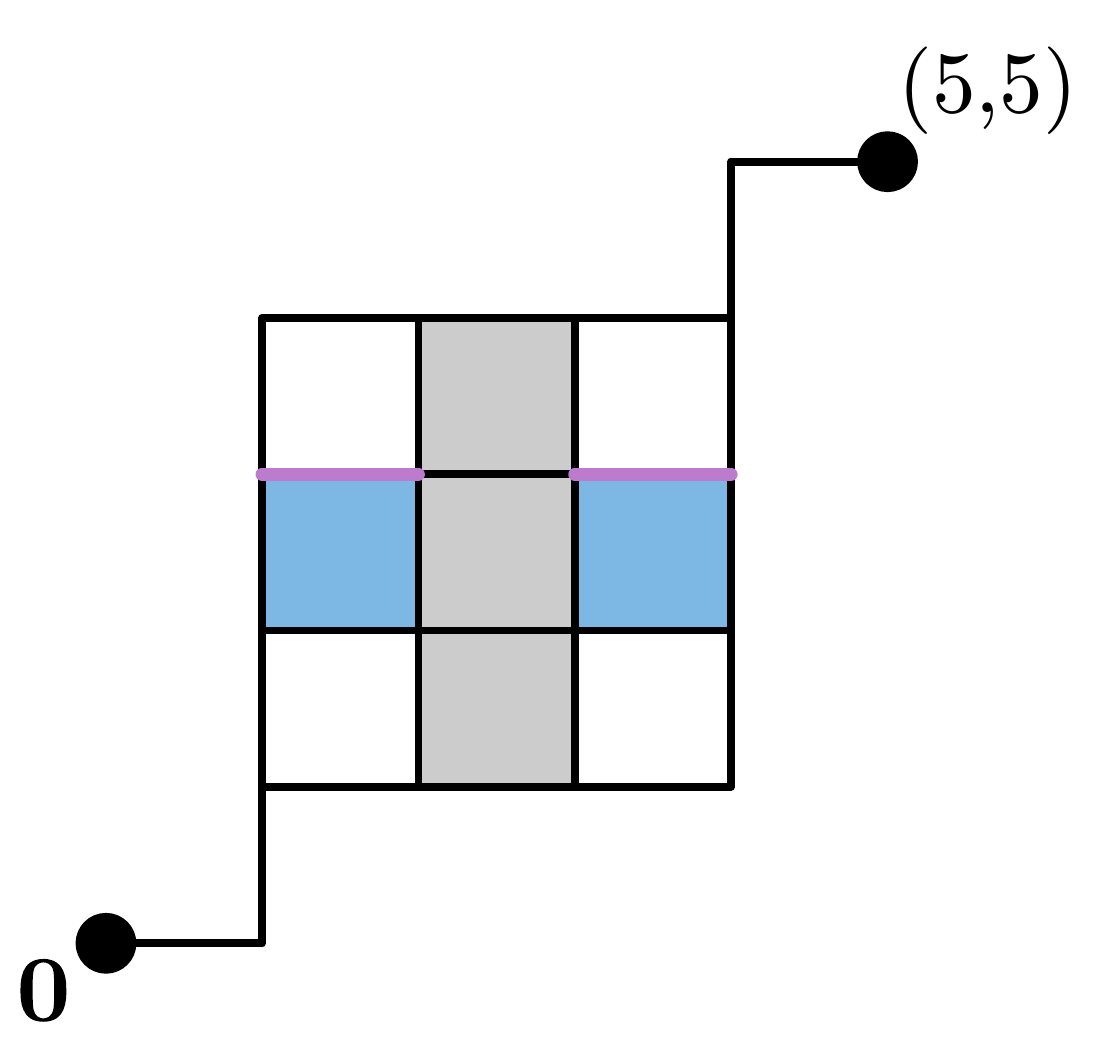}}%
    \subfloat[\label{subfig:window-collapse3}]{%
        \includegraphics[height=1in]{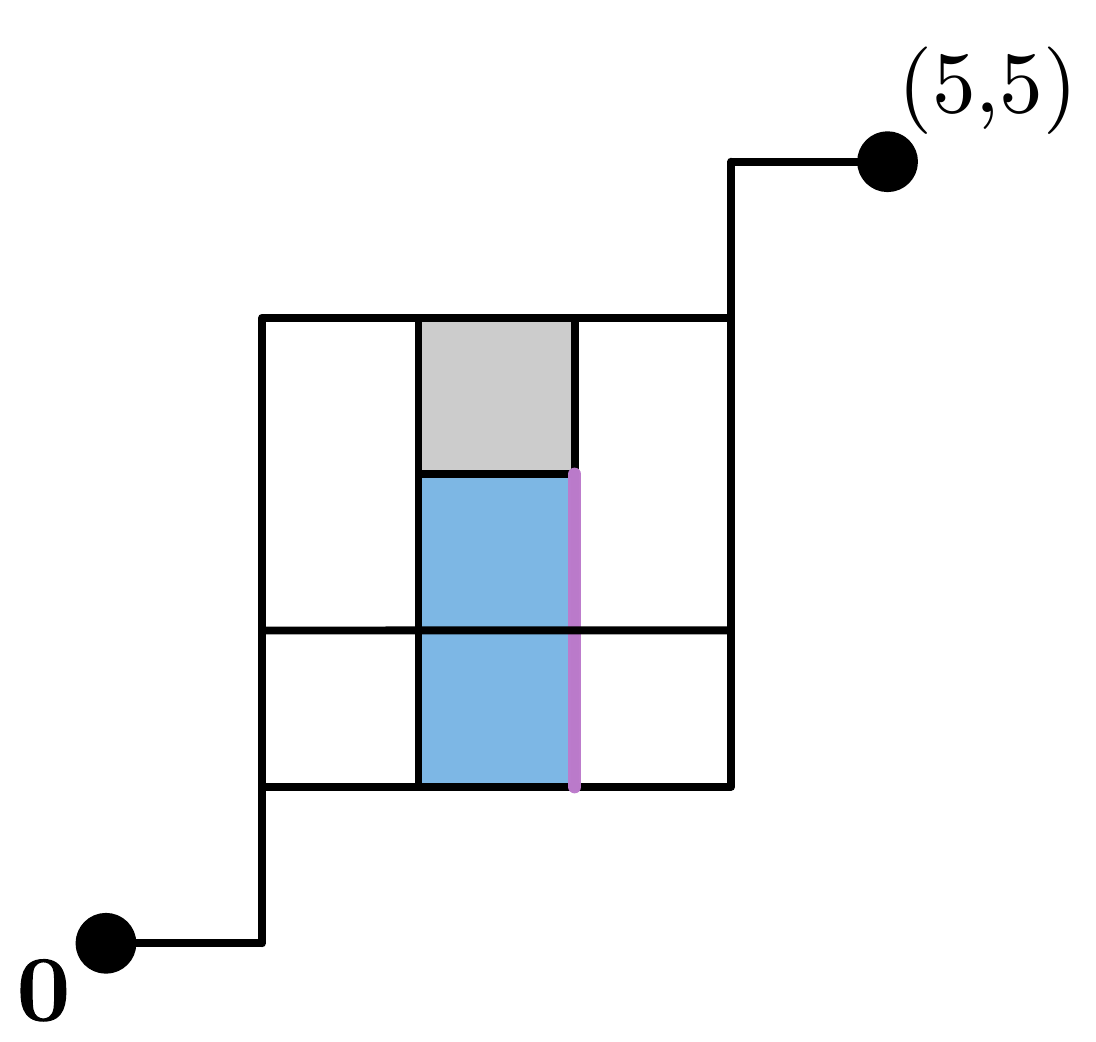}}\\
    \subfloat[\label{subfig:window-collapse4}]{%
        \includegraphics[height=1in]{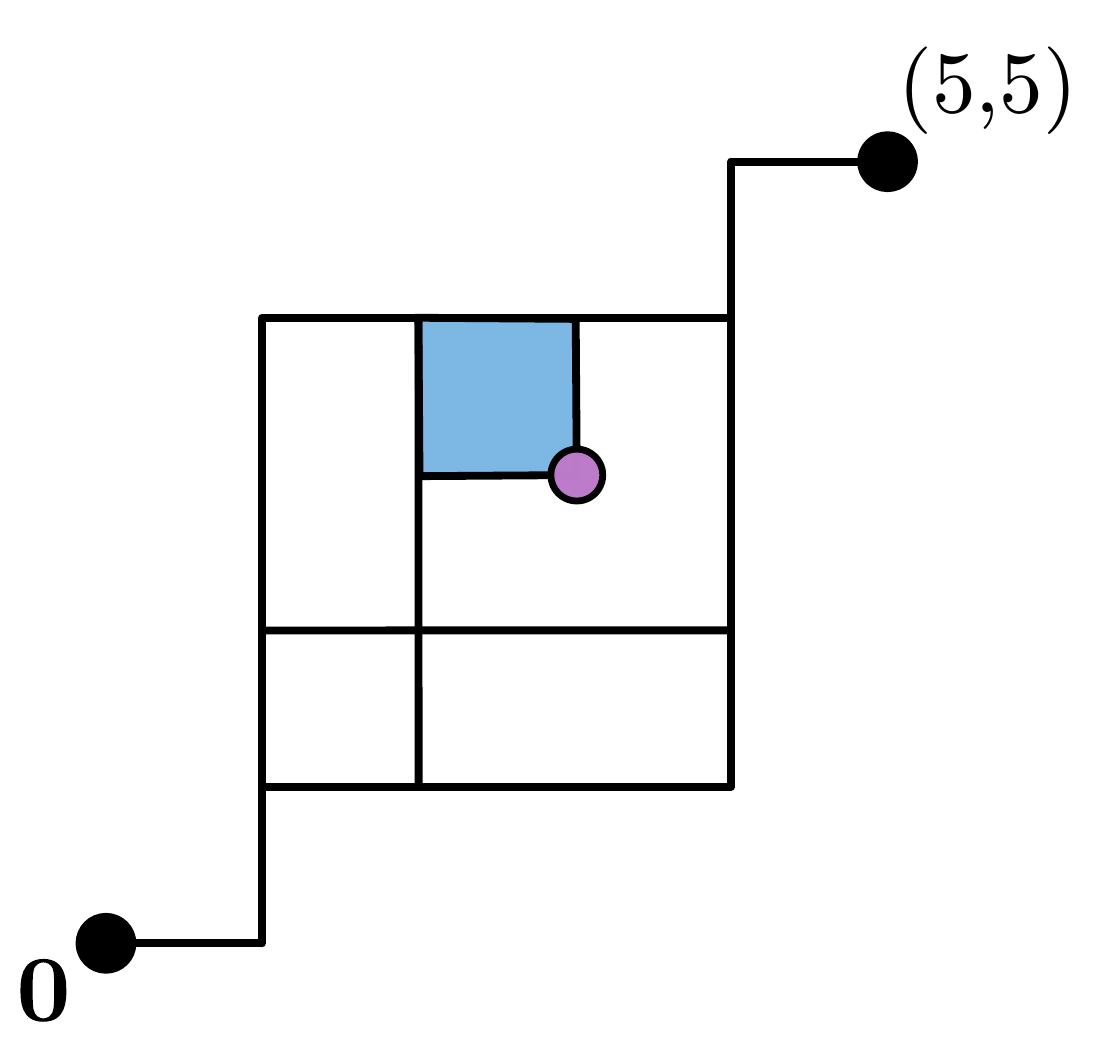}}%
    \subfloat[\label{subfig:window-collapse5}]{%
        \includegraphics[height=1in]{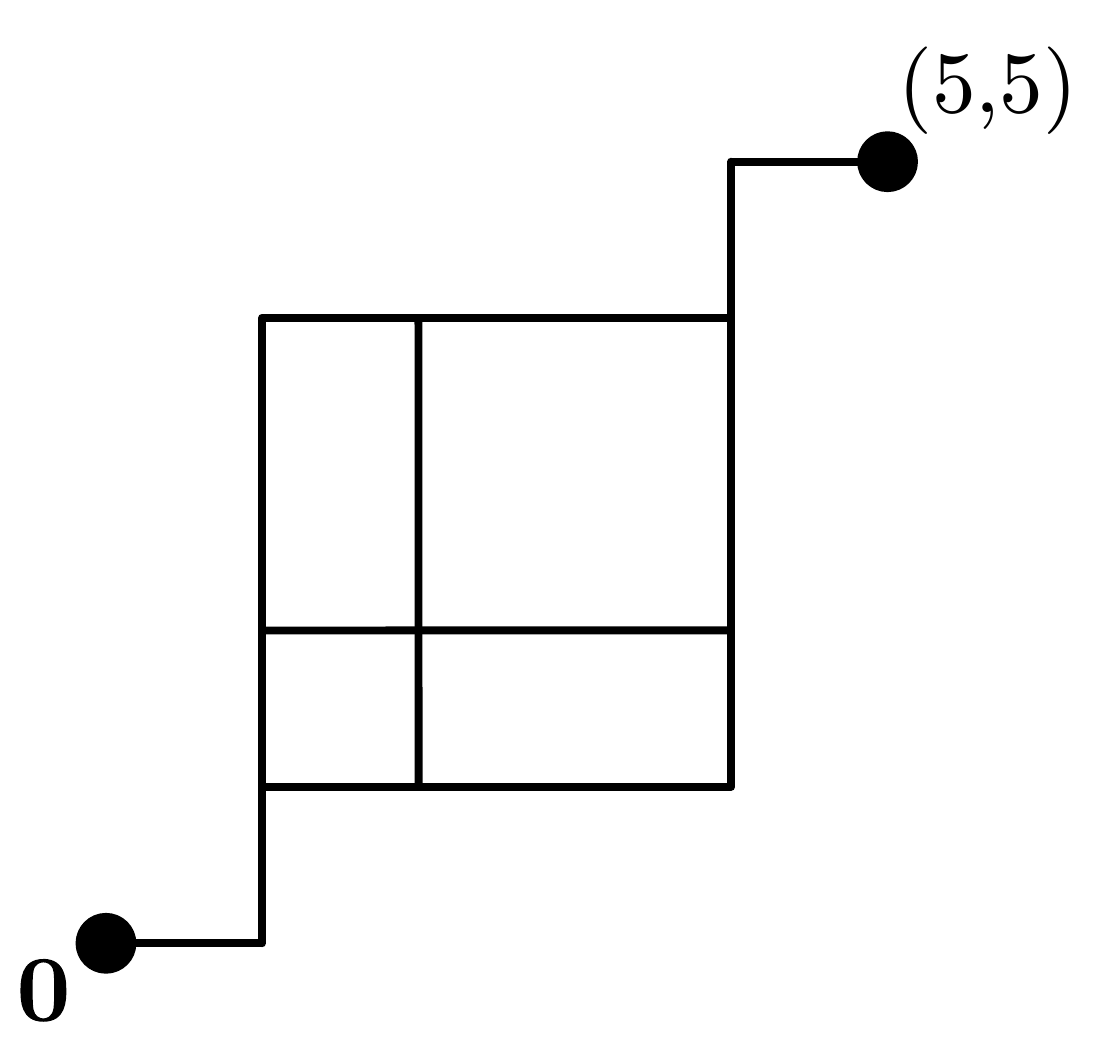}}%
\caption{Link-preserving DCs of the  window. A sequence of LPDCs is presented from \protect\subref{subfig:window1}-\protect\subref{subfig:window-collapse5}. The directed Euclidean cubical complex in \protect\subref{subfig:window-collapse2}
comes from performing several vertex LPDCs to remove the two-cubes along the border of $K$.
In \protect\subref{subfig:window-collapse2}-\protect \subref{subfig:window-collapse4}, the LPDC
pairs $(\tau, \sigma)$ are highlighted in purple and blue respectively. The result of the
sequence of LPDCs is a graph of vertices and edges that more clearly illustrates the
dihomotopy classes of dipaths in the dipath space.}
    \label{fig:window}
\end{figure}

    Let $K$ be the $5\times 5$ grid with the
    following two-cube interiors removed:  $[(1,1), (2,2)]$, $ [(3,1),(4,2)]$, $ [(1,3),(2,4)]$,  $[(3,3),(4,4)]$.
    See \figref{window}(a). $K$ has disconnected past links at the vertices
    $(2,2)$, $(4,2)$, $(2,4)$, $(4,4)$ so $K$ does not satisfy~\corref{contractible-collapse}
    or~\corref{connected-collapse}. Observe that $\di{K}_{\zero}^{(5,5)}$ has
    six connected components. We can perform a sequence of LPDCs that preserves
    the dihomotopy classes of dipaths between $\zero$ and $(5,5)$ at each step. First,
    we apply vertex LPDCs to remove the two-cubes along the border. Then we can
    apply four edge LPDCs and one vertex LPDC to get a graph of vertices
    and edges. This graph more clearly illustrates the six dihomotopy classes of dipaths in
    $\di{K}_{\zero}^{(5,5)}$.
\end{example}

In summary, LPDCs preserve the connectedness and/or contractability of dipath spaces
starting at the minimum vertex as long as $K$ has all reachable vertices and all dipath
spaces starting at the  minimum vertex in $K$ connected and/or contractible to begin with.
If $K$ does not have these properties, the first step could be to remove all unreachable vertices
and cubes before collapsing. In the next section, it will become clear that this will not suffice, if
the dipath spaces are not all connected or contractible.

\section{Discussion}\label{sec:discussion}

LPDCs preserve spaces of dipaths in many examples (see
\secref{pathspaces}), in particular, if they are all trivial in the sense of either all connected or all contractible and the directed Euclidean cubical complex is reachable from the minimum vertex.
However, LPDCs do not always preserve spaces of dipaths. We discuss some of those instances here.
One limitation of LPDCs is that the number of components may increase after an LPDC as we saw in \exref{bowlingball} or, as we see in
 \exref{badwindow}, they may decrease.

\begin{example}[A Sequence of LPDCs of the Window That Decreases the Number of Connected
    Components of the Dipath Space]\label{ex:badwindow}
    Consider $K$ as given in~\exref{window}. After applying vertex LPDCs that
    remove the two-cubes on the border of $K$, we can apply an LPDC to the edge~$[(2,4),(3,4)]$.
    Now $\di{K'}_{\zero}^{(5,5)}$ has five connected components;
    whereas, the dipath space~$\di{K}_{\zero}^{(5,5)}$ has six connected components. See \figref{badwindow}.
    \begin{figure}[th]
        \centering
        \subfloat[Initial Complex\label{subfig:window-badcollapse1}]{%
            \includegraphics[height=1.2in]{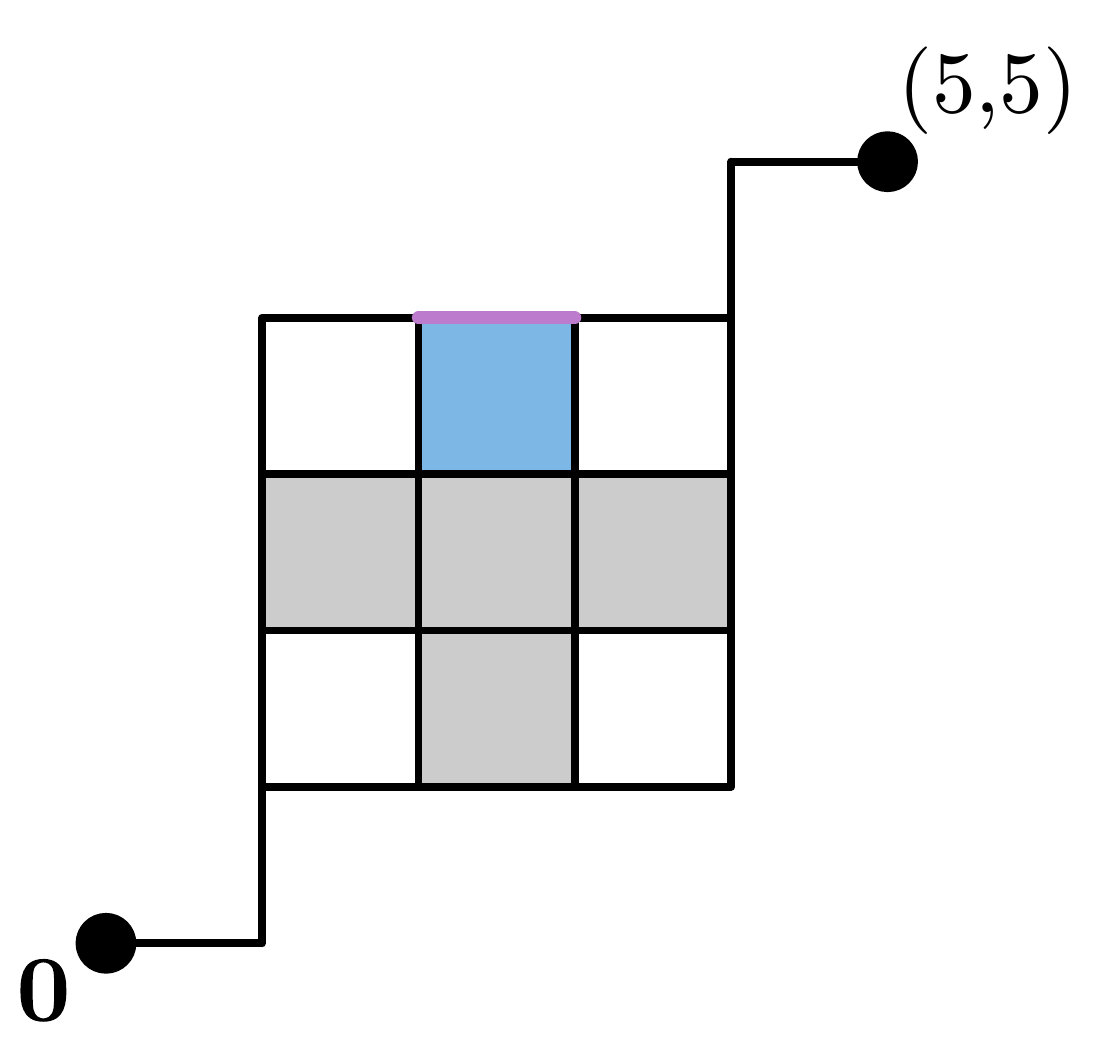}%
        }\hfil
        \subfloat[After Collapse\label{subfig:window-badcollapse2}]{%
            \includegraphics[height=1.2in]{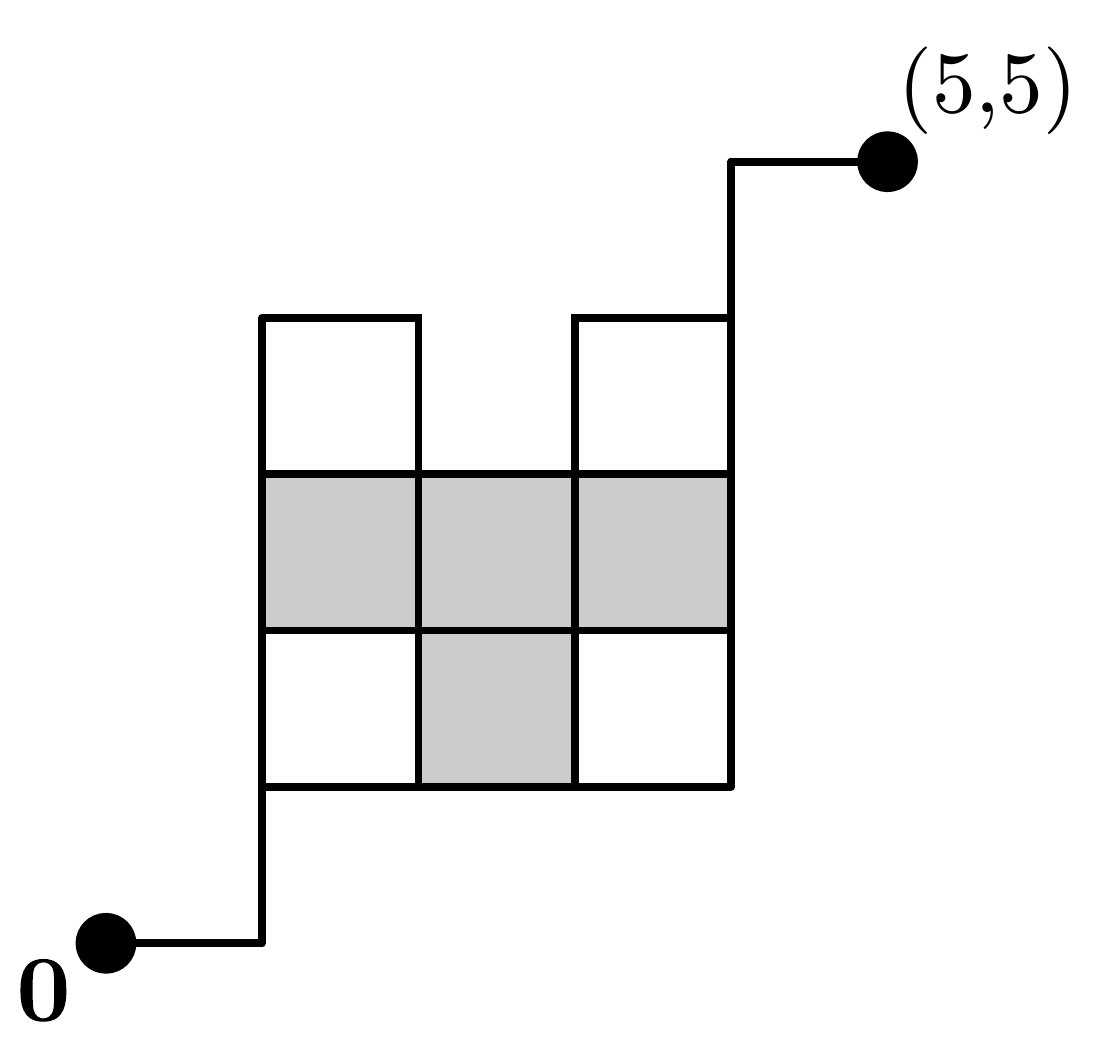}%
        }
        \caption{The LPDC for the window that changes dipath space. The
        LPDC of the edge $[(2,4),(3,4)]$ changes the dipath space
        between~$\zero$ to $(5,5)$ from having six connected components to five connected components.}
        \label{fig:badwindow}
    \end{figure}
    This example shows that there are both ``good" and ``bad" ways to apply a
    sequence of LPDCs to a directed Euclidean
    cubical complex. As illustrated in~\exref{window}, there exists a sequence of
    LPDCs that preserves the six connected components in $\di{K}_{\zero}^{(5,5)}$. However,
    if we perform a sequence of LPDCs that removes the edge $[(2,4),(3,4)]$ as in this
    example, then we get a directed Euclidean cubical complex that does not preserve the
    dihomotopy classes of dipaths in $\di{K}_{\zero}^{(5,5)}$.
\end{example}

\exref{badwindow} illustrates the need to investigate other properties if we want to preserve
dipath spaces
when performing an LPDC.

In \exref{bowlingball}, the problem was the existence of unreachable vertices. In \exref{badwindow},
the vertex $(2,4)$ is a \emph{deadlock} after the LPDC: only trivial dipaths
initiate from there;
whereas, before collapse, that was not the case. This seems to suggest that the introduction of
new deadlocks should not be allowed; in practice, this would require an
extra---but computationally easy---check on vertices of $\sigma$.

In the non-directed setting, if $K'$ is obtained from $K$ by collapsing a collapsing pair $(\tau, \sigma)$, then not only is the inclusion of $K'$ in $K$ a homotopy equivalence.
$K'$ is a deformation retract of $K$. The following example removes any hope of such a result in the directed setting:

\begin{example}[LPDC of the Four-Cube With No Directed Retraction
    to the Collapsed Complex]\label{ex:test}
    Let $(I^4,\mathcal{I}^4)$ be the standard unit four-cube. Let $\tau$ be the
    vertex $(1,1,0,0)$, and $\sigma$ be the cube~$[\zero,\one]$.
    Since $\tau$ is free and not the minimum vertex of $\sigma$,
    the pair~$(\tau, \sigma)$ is an LPDC pair.
    Thus, let $(K',\cubes{K'})$ be the collapsed complex. Next, we show that there is no directed
    retration, i.e., no directed map from~$I^4$ to~$K'$ that is the identity on~$K'$.

    Suppose, for a contradiction, that $f: I^4 \to K'$ is such a
    directed retraction.
    Let $\vecp_1=(0,1,0,0)$, $\vecp_2=(1,0,0,0)$, $\vecq_1=(1,1,1,0)$, and \mbox{$\vecq_2=(1,1,0,1)$.}
    By the product order on~$\R^4$, we have~$\vecp_1,\vecp_2 \preceq \tau$ and $\tau \preceq \vecq_1,\vecq_2$.
    Since the points~$\vecp_1$, $\vecp_2$, $\vecq_1$, and~$\vecq_2$ are vertices
    of~$I^4$ and are not equal to $\tau$, we also know
    that~$\vecp_1$,~$\vecp_2$,~$\vecq_1$, and~$\vecq_2$ are points in~$K'$.
    Since $f$ is a directed retraction, we have that~$\vecp_1 = f(\vecp_1)\preceq f(\tau)$
    and that~$\vecp_2 = f(\vecp_2)\preceq f(\tau)$.
    Similarly, we obtain that~$f(\tau)\preceq f(\vecq_1)=\vecq_1$
    and that~$f(\tau)\preceq f(\vecq_2)=\vecq_2$.

    Let $x_1,x_2,x_3,x_4 \in I$ such that $f(\tau) = (x_1,x_2,x_3,x_4)$.
    Then,
    \begin{align*}
        \vecp_1\preceq f(\tau)\Rightarrow x_2 &\geq 1
            \text{ and hence } x_2=1,\\
        \vecp_2\preceq f(\tau)\Rightarrow x_1 &\geq 1
            \text{ and hence } x_1=1,\\
        f(\tau)\preceq\vecq_1\Rightarrow x_4  &\leq 0
            \text{ and hence } x_4=0,\\
        f(\tau)\preceq \vecq_2\Rightarrow x_3 &\leq 0
            \text{ and hence } x_3=0.
    \end{align*}
    Thus,
    $f(\tau)=(1,1,0,0)=\tau$, which is not in $K'$ and hence a contradiction.
    In fact, this argument extends to $(I^k,\mathcal{I}^k)$ for $k \geq 4$.

    As further evidence that such a $(\tau,\sigma)$-collapse does not preserve
    the directed topology, consider the spaces of dipaths in $(I^4,\cubes{I}^4)$
    and $(K',\cubes{K}')$.  We would need dipaths in the original
    space to map to dipaths in the collapsed space.  However, notice that
    the dipath from $\vecp_1$ to
    $\vecq_1$ through~$\tau$ cannot be mapped to a dipath in~$(K',\cubes{K}')$.
\end{example}

We observe that vertex LPDCs appear to not introduce the problems of unreachability and deadlocks. These observations
lead us to suspect that studying unreachability, deadlocks, and vertex LPDCs can help us better understand
when LPDCs preserve and do not preserve dipath spaces between the minimum
and a given vertex. We leave this as future work.

In summary, we provide an easy criterion for determining when we have an LPDC
pair,
as well as discuss various settings for when LPDCs preserve spaces of dipaths.
Fully understanding when LPDCs preserve
spaces of dipaths between two given vertices is a step towards
developing algorithms that
compress directed Euclidean cubical complexes and preserve directed topology.


\begin{acknowledgement}
    This research is a product of one of the working groups at the
    Women in Topology (WIT) workshop at MSRI in November 2017. This workshop was
    organized in partnership with MSRI and the Clay Mathematics Institute, and was
    partially supported by an AWM ADVANCE grant (NSF-HRD 1500481).

    This material is based upon work supported by the US National Science Foundation under
    grant No.\  DGE 1649608 (Belton) and  DMS~1664858 (Fasy),
    as well as the Swiss National Science Foundation under grant No.\
    200021-172636 (Ebli).

    We thank the Computational Topology and Geometry (CompTaG) group at Montana State
    University for giving helpful feedback on drafts of this work.
\end{acknowledgement}

\bibliographystyle{plain}
\bibliography{sufficient-conditions-arxiv-ver2}

\begin{thebibliography}{10}

\bibitem{belton2020towards}
Robin Belton, Robyn Brooks, Stefania Ebli, Lisbeth Fajstrup, {Brittany Terese}
  Fasy, Catherine Ray, Nicole Sanderson, and Elizabeth Vidaurre.
\newblock Towards directed collapsibility.
\newblock In Bahar Acu, {Donatella } Danielli, Marta Lewicka, A.N. Pati,
  S.~{Ramanathapuram Vancheeswara}, and M.I. Teboh-Ewungkem, editors, {\em
  Advances in Mathematical Sciences}, Association for Women in Mathematics
  Series. Springer, United States, 2020.

\bibitem{BrownOn81}
Ronald Brown and Philip~J. Higgins.
\newblock On the algebra of cubes.
\newblock {\em Journal of Pure and Applied Algebra}, 21(3):233--60, 1981.

\bibitem{cohen2012course}
Marshall~M. Cohen.
\newblock {\em A Course in Simple-Homotopy Theory}, volume~10.
\newblock Springer Science \& Business Media, 2012.

\bibitem{dijkstra1977two}
Edsger~W. Dijkstra.
\newblock Two starvation-free solutions of a general exclusion problem.
\newblock 1977.
\newblock Manuscript EWD625, from the archives of UT Austin,

\bibitem{lisbeth}
Lisbeth Fajstrup, Eric Goubault, Emmanuel Haucourt, Samuel Mimram, and Martin
  Raussen.
\newblock {\em Directed Algebraic Topology and Concurrency}.
\newblock Springer, 1st edition, 2016.

\bibitem{FGR}
Lisbeth Fajstrup, Eric Goubault, and Martin Raussen.
\newblock Algebraic topology and concurrency.
\newblock {\em Theoretical Computer Science}, pages 241--271, 2006.

\bibitem{ghrist2010configuration}
Robert Ghrist.
\newblock Configuration spaces, braids, and robotics.
\newblock In {\em Braids: Introductory Lectures on Braids, Configurations and
  Their Applications}, pages 263--304. World Scientific, 2010.

\bibitem{goubault1995geometry}
Eric Goubault.
\newblock {\em The Geometry of Concurrency}.
\newblock PhD thesis, \'Ecole Normale Sup\'erieure, Paris, France, 1995.

\bibitem{grandis2003directed}
Marco Grandis.
\newblock Directed homotopy theory, {I}: The fundamental category.
\newblock {\em Cahiers de Topologie et G{\'e}om{\'e}trie Diff{\'e}rentielle
  Cat{\'e}goriques}, 44(4):281--316, 2003.

\bibitem{grandis2009diretced}
Marco Grandis.
\newblock {\em Directed Algebraic Topology: Models of Non-Reversible Worlds},
  volume~13 of {\em New Mathematical Monographs}.
\newblock Cambridge University Press, 2009.

\bibitem{hatcher}
A.~Hatcher.
\newblock {\em Algebraic Topology}.
\newblock Algebraic Topology. Cambridge University Press, 2002.

\bibitem{hoare}
Charles~A.R. Hoare.
\newblock Communicating sequential processes.
\newblock {\em Commun. ACM}, 21(8):666--677, August 1978.

\bibitem{jardine2002cubical}
John~F. Jardine.
\newblock Cubical homotopy theory: A beginning, 2002.
\newblock Technical report.

\bibitem{kozlovcombinatorial2007}
Dimitry Kozlov.
\newblock {\em Combinatorial Algebraic Topology}, volume~21 of {\em Algorithms
  and Computation in Mathematics}.
\newblock Springer-Verlag Berlin Heidelberg, 2008.

\bibitem{lachaud}
Jacques-Olivier Lachaud.
\newblock Cubical complex.
\newblock software
  \url{https://projet.liris.cnrs.fr/dgtal/doc/nightly/moduleCubicalComplex.html},
  2018.
\newblock part of the Toplogy package of the DGtal library.

\bibitem{RZ}
Martin Raussen and Krzysztof Ziemia{\'{n}}ski.
\newblock Homology of spaces of directed paths on {E}uclidean cubical
  complexes.
\newblock {\em Journal of Homotopy and Related Structures}, 9(1):67--84, Apr
  2014.

\bibitem{serre1951homologie}
Jean-Pierre Serre.
\newblock Homologie singuli{\`e}re des espaces fibr{\'e}s.
\newblock {\em Annals of Mathematics}, 54(3):425--505, 1951.

\bibitem{whitehead}
John H.~C. Whitehead.
\newblock Simplicial spaces, nuclei and $m$-groups.
\newblock {\em Proceedings of the London Mathematical Society},
  s2-45(1):243--327, 1938.

\bibitem{whitehead1950simple}
John H.~C. Whitehead.
\newblock Simple homotopy types.
\newblock {\em American Journal of Mathematics}, 72(1):1--57, 1950.

\bibitem{wisniewski2006towards}
Rafal Wisniewski.
\newblock Towards modelling of hybrid systems.
\newblock In {\em Proceedings of the 45th IEEE Conference on Decision and
  Control}, pages 911--916. IEEE, 2006.

\bibitem{ziemianski2016execution}
Krzysztof Ziemia{\'n}ski.
\newblock On execution spaces of {PV}-programs.
\newblock {\em Theoretical Computer Science}, 619:87--98, March 2016.

\end{thebibliography}

\end{document}